\documentclass[11pt,a4paper]{article}

\usepackage{graphicx}
\usepackage{amsmath}
\usepackage{amssymb}
\usepackage{enumerate}
\usepackage{amsthm}
\usepackage[T1]{fontenc}
\usepackage[latin1]{inputenc}
\let\BFseries\bfseries\def\bfseries{\BFseries\mathversion{bold}} % formulas in headings bold
\newtheorem{thm}{Theorem}

\newtheorem{lemma}[thm]{Lemma}
\newtheorem{rem}[thm]{Remark}

\theoremstyle{definition}

\newcommand{\pr}[2]{\mathbb{P}_{#1}\left(#2\right)}

\newcommand{\PRO}{\mathbb{P}}
\newcommand{\ER}{\mathbb{E}}
\newcommand{\eps}{\varepsilon}
\newcommand{\FF}{{\mathcal F}}

\begin{document}

\title{First passage times of L\'evy processes over \\ \vspace*{0.4 cm}  a one-sided moving boundary \\}

\author{\renewcommand{\thefootnote}{\arabic{footnote}} {\sc Frank Aurzada}\footnotemark[1], \renewcommand{\thefootnote}{\arabic{footnote}}{\sc Tanja Kramm}\footnotemark[1],  \hspace{2pt} and 
\renewcommand{\thefootnote}{\arabic{footnote}}{\sc Mladen Savov}\footnotemark[2]}
\bigskip \bigskip \bigskip \bigskip \bigskip \bigskip\bigskip \bigskip\bigskip \bigskip \bigskip \bigskip \bigskip \bigskip \bigskip \bigskip\bigskip \bigskip
\bigskip
\bigskip
\bigskip
\bigskip
\bigskip
\bigskip
\bigskip
\date{\today}

\footnotetext[1]{
Technische Universit\"at Braunschweig, Institut f\"ur Mathematische Stochastik, Pockelsstra\ss e 14, 38106 Braunschweig, Germany
{\sl f.aurzada@tu-braunschweig.de},
{\sl tanjakramm@gmail.com}
}
\footnotetext[2]{
University of Reading, Department of Mathematics and Statistics, Whiteknights, PO Box 220, Reading RG6 6AX, UK
{\sl m.savov@reading.ac.uk}
}

\date{\today}

\maketitle
\begin{abstract}
We study the asymptotic behaviour of the tail of the distribution of the first passage time of a L\'evy process over a one-sided moving boundary. 
Our main result states that if the boundary behaves as $t^{\gamma}$ for large $t$ for some $\gamma<1/2$ then the probability that the process stays below the boundary behaves asymptotically as in the case of a constant boundary. We do not have to assume Spitzer's condition in contrast to all previously known results. Both positive ($+t^\gamma$) and negative ($-t^\gamma$) boundaries are considered. 

% To this aim, we develop a new technique using an iteration method to reduce the exponent $\gamma$ of the boundary in each step such that the boundary eventually turns into a constant boundary. 
 
These results extend the findings of \cite{GreNov} and are motivated by results in the case of Brownian motion, for which the above result was proved in \cite{Uch}.

\end{abstract}
% 
% \bigskip
% \bigskip

\vfill

%Research supported by DFG Emmy Noether programme.
\newpage

\section{Introduction}
\subsection{Statement of the problem and summary of results}
We consider the one-sided exit problem with a moving boundary. In the literature, this problem is known by a variety of names, e.g.\ \textit{one-sided barrier problem}, \textit{boundary crossing problem}, \textit{persistence probabilities}, and \textit{first passage time problem}. For a stochastic process $(X(t))_{t\geq0}$ and a function $f:\mathbb{R}_+ \rightarrow \mathbb{R}$, the so-called moving boundary, the question is to determine the asymptotic rate of the probability
\begin{align}\label{problem}
\PRO \left(  X(t) \leq f(t) , \text{ } 0 \leq t \leq T  \right), \qquad \text{as } T \rightarrow \infty.
\end{align}
If this probability is asymptotically polynomial of order $-\delta$ (e.g.\ if it is regularly varying with index $-\delta$), the number $\delta$ is called the \textit{survival exponent} or \textit{persistence exponent}. If the function $f$ is constant then we are in the classical framework of first passage times over a constant boundary.

This problem is a classical question, which is relevant in a number of different applications, a recent overview is presented in \cite{AurSim} and \cite{lishao}. Let us first review some results involving Brownian motion and L\'evy processes and then summarise the contribution of this paper.
\medskip

% Let $\tau_f$ denote the first passage time across the boundary $f$ for the c\`{a}dl\`{a}g process $(X(t))_{t\geq 0}$, i.e.
% \begin{align*}
%  \tau_f = \inf \{ t\geq 0: X(t) > f(t)\}. 
% \end{align*}
% Since $\PRO (\tau_f > T) = \PRO \left(  X(t) \leq f(t) , \text{ } 0  \leq  t \leq T  \right) $, our problem is equivalent to the study of the asymptotic behaviour of the distribution of the first passage time. 
% \medskip

In the case that $X$ is a Brownian motion, $\sup_{0\leq t\leq T} B_t$ has the same law as $|B_T|$, by the reflexion principle. From this, everything concerning any \textit{constant} boundary is deduced easily and, in the above terminology, the survival exponent equals $1/2$. However, even for Brownian motion, the question involving \textit{moving} boundaries (\ref{problem}) is already non-trivial. It is studied by \cite{Uch,Gae,jenler,Sal,Nov,Novneu,AurKra} in different ways. Independently of each other \cite{Gae} and \cite{Uch} state an integral test for the boundary $f$, for which the survival exponent remains $1/2$. More precisely, they prove under some additional regularity assumptions that
\begin{align}\label{eqn: necsufBM}
 \int_1^{\infty} |f(t)| t^{-3/2}dt < \infty \Longleftrightarrow  \PRO (X(t) \leq f(t), \text{ } 0\leq t \leq T) \approx  T^{-1/2}, \text{ as } T \rightarrow \infty.
\end{align}
Here and below we use the following notation for strong and weak asymptotics. We write $f \lesssim g$ if $\limsup_{x \to \infty} f(x)/g(x) < \infty$ and $f\approx g$ if  $f \lesssim g$ and $g \lesssim f$. Furthermore, $f \sim g$ if $f(x)/g(x) \rightarrow 1$ as $x \rightarrow \infty$.
\medskip

 For L\'evy processes, the study of the first passage time distribution over a \textit{constant} boundary is a classical area of reasearch. The results follow from fluctuation theory; e.g.\ \cite{Rog} shows that the survival exponent is equal to $\rho \in (0,1)$ if $X$ satisfies Spitzer's condition with $\rho \in (0,1)$, that is, $\PRO (X(t) >0) \rightarrow \rho$, as $t \rightarrow \infty$ (cf.\ \cite{BerDon}). Generally, the assumption of Spitzer's condition appears in the majority of works on this subject; we stress that the technique in this paper is independent of Spitzer's condition. Similar arguments as for L\'evy processes were already used for random walks with zero mean (see e.g.\  \cite{feller}). If the process does not necessarily satisfy Spitzer's condition, various results were obtained for a constant boundary by \cite{Bal,BerDon2,Bor1,Bor2,DenShn,Don2,SupLP}.  

In this paper, we consider L\'evy processes $(X(t))_{t\geq0}$ with triplet $(\sigma^2, b, \nu)$ and consider \textit{moving} boundaries. We focus on the following question: For which functions $f$ does the asymptotic behaviour of the non-exit probability for a constant boundary, i.e.
\begin{align}\label{assum}
\PRO \left(  X(t) \leq 1 , \text{  } 0 \leq t \leq T  \right) = T^{-\delta + o(1)}, \quad \text{as } T \rightarrow \infty,
\end{align}
imply the same asymptotic behaviour for $(\ref{problem})$?

\bigskip
Let us now summarise our results and compare to previously known ones. For this purpose,  let us look for a moment at functions $f(t)= 1 \pm t^{\gamma}$, $\gamma\geq 0$, for simplicity.

\smallskip \textbf{Negative boundary $1-t^\gamma$:} Our first main result, Theorem \ref{falling}, says that if $\nu(\mathbb{R}_-) >0$ and (\ref{assum}) hold then
 \begin{align*}
  \gamma < \frac{1}{2} \quad \Rightarrow \quad \PRO (X(t) \leq 1- t^{\gamma}, \text{ } 0 \leq  t \leq T ) = T^{-\delta +o(1)}, \quad \text{as } T \rightarrow \infty.
 \end{align*}
Note that we do not require any conditions on the left or right tail of the L\'evy measure, neither Spitzer's condition. Negative results (i.e.\ situations where the survival exponent does change) are given in \cite{mogpec,GreNov}. Results similar to those for Brownian motion are only available under such heavy assumptions as bounded jumps from above or $X$ satisfying Cram\'{e}r's condition, see \cite{Nov2} or \cite{Novdis}.

\smallskip
\textbf{Positive boundary $1+t^\gamma$:} Our second main result, Theorem \ref{mblp}, says that assuming that $\nu(\mathbb{R}_+) >0$, $\nu(\mathbb{R}_-) >0$,  and (\ref{assum}) hold we have
\begin{align*}
 \gamma < \frac{1}{2} \quad \Rightarrow \quad \PRO (X(t) \leq 1+ t^{\gamma}, \text{ } 0 \leq t \leq T ) = T^{-\delta+o(1)}, \quad \text{as } T \rightarrow \infty.
\end{align*}
Again, no conditions for the left or right tail of the L\'evy measure are needed. On the other hand, assuming that Spitzer's condition holds with $\rho \in (0,1)$, the result of \cite{GreNov}  states that
\begin{align*}
 \gamma < \rho \quad \Rightarrow \quad \PRO (X(t) \leq 1+ t^{\gamma}, \text{ } 0 \leq t \leq T ) \sim  T^{-\rho} \ell (T)  \quad \text{as } T \rightarrow \infty,
\end{align*}
where $\ell$ is a slowly varying function. Hence, we improve the result of \cite{GreNov} when $\rho <\tfrac{1}{2}$ or when $X$ does not satisfy Spitzer's condition. Note that \cite{GreNov} determines the exact asymptotics; consequently, \cite{GreNov} gives a more precise result for $\gamma < \rho$.

\medskip
The main contributions of this paper can be summarised as follows:
\begin{itemize}
 \item We show a way to transfer results for a constant boundary (\ref{assum}) to moving boundaries. In this connection, Spitzer's condition is not required at any point in our arguments.
 \item In the simplified case, $f(t)= 1 \pm t^{\gamma}$, we obtain the same result as for Brownian motion (see \cite{Uch}). Intuitively, this follows from the fact that a L\'evy process allows more (large) fluctuations than Brownian motion and can thus follow a boundary at least as well as Brownian motion.
 \item This paper is meant to be a first attempt to find necessary and sufficient conditions for the boundary $f$ (in the simplified case, that is, find optimal $\gamma$) such that the non-exit probabilities for constant and moving boundaries have the same asymptotic behaviour.
\end{itemize}

On the downside, we can only control the polynomial order term of the probability. Contrary, for constant boundaries more precise results can be obtained -- often, the probability in question is shown to be regularly varying. We stress that the techniques used for that type of results do not seem applicable to moving boundaries. The reason is that, unlike in the constant boundary case and for a small class of very specific decreasing moving boundaries (cf.\ \cite{mogpec}), no factorization identities are known yet for moving boundaries. Our results are a first attempt to approach the problem and to find different effects that allow different boundaries.
\medskip

Let us mention that related topics have been discussed like the moments (\cite{DonMal,Gut,Rot}), the finiteness (\cite{DonMal2}), and the stability (\cite{GrifMal}) of the first passage time. Furthermore,  L\'evy processes and stochastic boundaries (\cite{VO}) % as well as diffusion processes with a moving boundary (\cite{BorDow1} and \cite{BorDow})
are discussed in the literature. 
\medskip

We proceed this paper by formally introducing our main results in Section \ref{mainresult}. There, we also present the main idea of the proofs.  The proof of Theorem \ref{falling}, the case of negative  boundaries, is given in Section \ref{proofthm2}, whereas Section \ref{proof1} contains the proof for positive boundaries, Theorem \ref{mblp}. For reasons of clarity and readability some auxiliary lemmas are combined in Section \ref{helpresult} and may be of independent interest.

%%%%%%%%%%%%%%%%%%%%%%%%%%%%%%%%%%%%%%%%%%%%%%%%%%%%%%%%%%%%%%%%%%%%%%%%%%%%%%%%%%%%%%%%%%%%%%%%%%%%%%%%%%%%%%%%%%%%%%%%%%%%%%%%%%%%%%%%%%%%%%%%%%%%%%%%%%%%%%%%%%%%%%%%%%%%%
\subsection{Main results}\label{mainresult}
We study the one-sided exit problem with moving boundaries for a L\'evy process denoted by $(X(t))_{t\geq 0}$. L\'evy processes possess stationary and independent increments and almost surely right continuous paths (see \cite{bertoin}, \cite{sato}). By the L\'evy-Khintchine formula, the characteristic function of a marginal of a L\'evy process $(X(t))_{t\geq 0}$ is given by
\begin{align*}
 \ER \left( e^{iuX(t)}  \right) = e^{t \Psi(u)}, \quad \text{for every } u \in \mathbb{R},
\end{align*}
where
\begin{align}\label{char}
 \Psi (u) =   i b u - \frac{\sigma^2}{2} u^2+ \int_{\mathbb{R} } (e^{iux}-1 - \mathbf{1}_{\{ |x| \leq 1  \}} iux) \nu (dx),
\end{align}
for parameters  $\sigma^2 \geq 0 $,  $b \in \mathbb{R}$, and a positive measure $\nu $ concentrated on $\mathbb{R}\backslash \{0\}$, called L\'evy measure, satisfying 
\begin{align*}
\int_{\mathbb{R}} (1 \wedge x^2) \nu (dx) < \infty.
\end{align*}
For a given triplet $(\sigma^2, b, \nu)$ there exists a L\'evy process $(X(t))_{t\geq 0}$ such that (\ref{char}) holds, and its distribution is uniquely determined by its triplet. We call $(X(t))_{t\geq 0}$  a $(\sigma^2, \nu)$-L\'evy martingale if (\ref{char}) is equal to 
\begin{align}\label{marchar}
 \Psi (u) = -\frac{\sigma^2}{2} u^2+ \int_{\mathbb{R}} (e^{iux}-1 - iux) \nu (dx)
\end{align}
for a measure $\nu$ satisfying $\int (|x| \wedge x^2) \nu(dx)<\infty$. It is a martingale in the usual sense.

\medskip

We can now formulate our first main result, which corresponds to  the one-sided exit problem with a {\it negative boundary}.
\begin{thm}\label{falling}
 Let $X$ be a L\'evy process with triplet $(\sigma^2,b,\nu)$ where $\nu (\mathbb{R}_-) >0$. Let $f \colon \mathbb{R}_+ \rightarrow \mathbb{R}_+$  be a differentiable, non-decreasing function such that $f(0) <1 $,
 $\int_1^\infty f'(s)^2 ds < \infty$, and $f'(t) \searrow 0$, for $t\rightarrow \infty$. Let $\delta > 0$. If 
\begin{align}\label{fallingass}
 \PRO (X(t) \leq 1, \text{ } 0\leq t \leq T) = T^{-\delta+ o(1)},\quad \text{as } T \rightarrow \infty
\end{align}
 holds then 
\begin{align}\label{resultthm1}
 \PRO (X(t) \leq 1- f(t), \text{ } 0\leq t \leq T) = T^{-\delta+ o(1)}, \quad \text{as } T \rightarrow \infty.
\end{align}
\end{thm}

The following theorem corresponds to the one-sided exit problem with a {\it positive boundary}.
\begin{thm}\label{mblp}
 Let $X$ be a L\'evy process with triplet $(\sigma^2 ,b,\nu)$ where $\nu (\mathbb{R}_+) >0$ and $\nu (\mathbb{R}_-) >0$. Let $f \colon \mathbb{R}_+ \rightarrow \mathbb{R}_+$  be a differentiable, non-decreasing function such that  $\int_1^\infty f'(s)^2 ds < \infty$ and $\sup_{s\geq 1} |f'(s)| < \infty$. Let $ \delta > 0$. If 
\begin{align}\label{mblpass}
 \PRO (X(t) \leq 1, \text{ } 0\leq t \leq T) = T^{-\delta+ o(1)}, \quad \text{as } T \rightarrow \infty
\end{align}
holds then 
\begin{align*}
 \PRO (X(t) \leq 1+ f(t), \text{ } 0\leq t \leq T) = T^{-\delta+ o(1)}, \quad \text{as } T \rightarrow \infty.
\end{align*}
\end{thm}

The proofs of these theorems are given in Section~\ref{proofthm2} and~\ref{proof1}, respectively, and the ideas will be sketched below.

Let us give a few comments on these results.
\begin{rem}
  In Theorem \ref{falling}  (Theorem \ref{mblp}, respectively), the assumption that there are negative (positive, respectively) jumps is an essential part of our technique. We will ``compensate'' the (negative/positive) boundary by (negative/positive) jumps and thus reduce the problem to the constant boundary case.
\end{rem}

\begin{rem}
In both Theorems, the regularity conditions on the function $f$ are for technical purposes only. Trivially, both Theorems are also valid for a less regular function $g$ if there is a function $f$ satisfying the conditions in Theorem \ref{falling} (Theorem \ref{mblp}, respectively) such that $g(s) \leq f(s) $, for all $s \geq 0$. The important property of the function $f$ is its asymptotic behaviour at infinity,
$$
\int_1^\infty f'(t)^2 d t < \infty,
$$
which is a slightly weaker assumption than Uchiyama's integral test (\ref{eqn: necsufBM}).
\end{rem}

\begin{rem}\label{rem:negjumps}
 The assumption of negative jumps in Theorem \ref{mblp} seems to be of technical matter. Different assumptions exist in order to replace the assumption of negative jumps such as the assumption that
\begin{enumerate}[(a)]
 \item the renewal function $U$ of the ladder height process satisfies  $U((\ln T)^5) \leq T^{o(1)}$, or
 \item there is a $T_0 \in (1,T^{o(1)})$ depending on $T$  such that $\PRO (X(T_0) \leq - (\ln T)^{5}) \geq T^{o(1)}$.
 \end{enumerate}
 See Remark \ref{rem:detailsnegjumps} below for a detailed discussion.
\end{rem}

\begin{rem}\label{rem: Spitzer}
  The assumption of equation (\ref{fallingass})/(\ref{mblpass}) is associated with Spitzer's condition. 
Recall that (cf.  \cite{Rog} or  \cite{bertoin}, Theorem 18) Spitzer's condition holds with $ \rho \in (0,1)$ if and only if the probability in (\ref{fallingass})/(\ref{mblpass}) is regularly varying with index $-\rho$. Note that the class of L\'evy processes satisfying assumption (\ref{fallingass})/(\ref{mblpass}) is strictly larger than the class of L\'evy processes satisfying  Spitzer's condition (see \cite{DenShn}, or  \cite{BerDon2,Don2} for a discrete-time version). For instance, L\'evy processes where $\ER X (1) \in (0,\infty)$ and the left tail of the L\'evy measure is  regularly varying with index $-c$, $c >1$, satisfy assumption (\ref{fallingass})/(\ref{mblpass}) with $\delta = c$, but not Spitzer's condition with $\rho \in (0,1)$.
\end{rem}

Let us come back to the question posed in (\ref{assum}), whether necessary and sufficient conditions on the boundary exist for which the survival exponent stays the same compared to the case of a  constant boundary. More precisely, let $\delta > 0$,  $\alpha_+ := \sup\{r\geq 0: \ER \left( (X(1)^+)^{r} \right) < \infty\}$ and $\alpha_- := \sup\{r\geq 0: \ER \left( (X(1)^-)^{r} \right) < \infty\}$. Because of the present results and previously known ones (e.g.\ \cite{DonMal2}, \cite{GreNov}, and \cite{mogpec}) it seems to be reasonable to expect that (\ref{fallingass}) implies
\begin{align*}
 \gamma < \max \left\{  \tfrac{1}{2}, \tfrac{1}{\alpha_-}   \right\} \Longleftrightarrow \PRO (X(t) \leq 1 - t^{\gamma}, \text{ } 0\leq t \leq T) = T^{-\delta+ o(1)}.
\end{align*} 
We have shown sufficiency of $ \gamma <  \tfrac{1}{2}$.

In the same way, one might also expect that (\ref{mblpass}) implies
\begin{align*}
 \gamma < \max \left\{  \tfrac{1}{2}, \tfrac{1}{\alpha_+}, \tfrac{1}{\alpha_-} \right\} \Longleftrightarrow \PRO (X(t) \leq 1 + t^{\gamma}, \text{ } 0\leq t \leq T) = T^{-\delta+ o(1)}.
\end{align*} 
Combining our results with \cite{GreNov} (who assume Spitzer's condition with $\rho\in(0,1)$) shows sufficiency of $ \gamma < \max \left\{  \tfrac{1}{2}, \rho \right\}$. Recall that for any L\'evy process belonging to the domain of attraction of a strictly stable process with index $\alpha \in (0,2)$ we have $\rho \leq \max \{\tfrac{1}{\alpha_+}, \tfrac{1}{\alpha_-}\}=\tfrac{1}{\alpha}$ (cf.\ \cite{Zolo}).
\medskip

We conclude this section by presenting a sketch of the proof of Theorem \ref{falling}.  For this purpose, we need the definition of an additive process. This class of processes consists of time-inhomogeneous processes which have independent increments and start at $0$ (see \cite{sato}). The triplet is given by $(\sigma^2, f_X (t), \Lambda_X (dx,dt))$, $f_X \in C[0,\infty)$ where $f(0)=0$, $\sigma \geq 0$, and $\Lambda_X$ is a measure on $\mathbb{R}\times [0,T]$.  

\textbf{Sketch of the proof of Theorem \ref{falling}:} Note that the upper bound is trivial since $f$ is positive.
For the lower bound our main idea is to find an iteration method to reduce the exponent of the boundary in each step such that eventually the boundary turns into a constant boundary. In each iteration step, we start with a change of measure compensating the boundary $f$ by negative jumps. Then, we get an additive process which has the following triplet $\left(\sigma^2, b\cdot s, (1+f'(s)|x|/m \mathbf{1}_{\{x\in A\}})ds\nu(dx) \right)$, where $A\subseteq [-1,0)$ and $m$ are suitably chosen. This process can be represented as $X(\cdot) +Z(\cdot)$, where $X$ is the original L\'evy process and $Z$ has the triplet $(0,0, f'(s)|x|/m \mathbf{1}_{\{x\in A\}}ds\nu(dx))$. This approach implies the estimate
\begin{align*}
 \PRO (X(t) \leq 1 - f(t) , \text{ } t\leq T)\geq  \PRO (X(t) + Z(t) \leq 1, \text{ } t\leq T) \cdot  e^{-c \sqrt{\ln T}}.
\end{align*}
The term $\exp\left(- c\, \sqrt{\ln T}  \right)$ represents the cost of changing the measure. A homogenization yields a L\'evy process $\tilde{Z}$ with $Z(\cdot)\overset{d}{=} \tilde{Z}(f(\cdot))$ and triplet $(0,0,|x|/m \mathbf{1}_{\{x\in A\}}\nu(dx))$. Since $\tilde{Z}$ is a L\'evy martingale with some finite exponential moment, 
% a coupling with a Brownian motion can be applied to estimate $\PRO (X(t) + Z(t) \leq 1, \text{ } t\leq T)$ essentially by
% \begin{align*}
%  \PRO (X(t) + B(f(t)) \leq 1, \text{ } t\leq T) .
% \end{align*}
%  Next, we take advantage of properties of Brownian motion to complete this step and 
we can finally estimate $\PRO (X(t) + \tilde{Z} (f(t) ) \leq 1, \text{ } t\leq T)$  by $\PRO (X(t) \leq 3 - f(t)^{2/3}, \text{ } t\leq T)$ giving essentially
\begin{align*}
 \PRO (X(t) \leq 1 - f(t) , \text{ } t\leq T)\geq \PRO (X(t) \leq 3 - f(t)^{2/3}, \text{ } t\leq T)\, e^{-c \sqrt{\ln T}}.
\end{align*}
 This procedure is repeated until $f(t)^{(2/3)^{n}}\leq 2$. Then, the asymptotic behaviour of $\PRO (X(t) \leq 3 - f(t)^{(2/3)^n}, \text{ } t\leq T) $ follows from (\ref{assum}). Hence, through an $n$-times iteration of these steps the  survival exponent in (\ref{problem}) is obtained with the help of (\ref{assum}) since $n$ is of order $\ln \ln T$. A similar approach is used in the proof of Theorem \ref{mblp}. Here, the upper bound is proved through an iteration method. 

%%%%%%%%%%%%%%%%%%%%%%%%%%%%%%%%%%%%%%%%%%%%%%%%%%%%%%%%%%%%%%%%%%%%%%%%%%%%%%%%%%%%%%%%%%%%%%%%%%%%%%%%%%%%%%%%%%%%%%%%%%%%%%%%%%%%%%%%%%%%%%%%%%%%%%%%%%%%%%%%%%%%%%%%%%%%%

%%%%%%%%%%%%%%%%%%%%%%%%%%%%%%%%%%%%%%%%%%%%%%%%%%%%%%%%%%%%%%%%%%%%%%%%%%%%%%%%%%%%%%%%%%%%%%%%%%%%%%%%%%%%%%%%%%%%%%%%%%%%%%%%%%%%%%%%%%%%%%%%%%%%%%%%%%%%%%%%%%%%%%%%%%%%%%%%%%%%%%%%%%%%%%
\section{Auxiliary results}\label{helpresult}
\subsection{Technical tools regarding the boundary and Girsanov transform for additive processes}
% In this section, we show that it is sufficient for the proofs to consider only non-decreasing functions $f$ during the further progress of this work. Furthermore, we indicate some useful properties of the function $f$.
% 
% Remark \ref{noinc} shows that it is sufficient to prove Theorem  \ref{falling}  and Theorem \ref{mblp} for functions $f$ which are  non-decreasing.
% 
% 
% \begin{rem}\label{noinc}
% We do not require the restriction that the function is non-decreasing to prove the upper bound of Theorem \ref{falling} (respectively, the lower bound of Theorem \ref{mblp}). The lower (respectively, upper) bound can be deduced from the case where the function $f$ is non-decreasing through adequate estimates. For this purpose, define $ M(t) :=  \sup_{0 \leq s\leq t} f(s)$. We proceed for functions satisfying the assumptions of Theorem \ref{falling} as follows:
% \begin{align*}
%   \PRO (X(t) \leq 1- f(t), \text{ } 0\leq t \leq T) \leq  \PRO (X(t) \leq 1- M(t), \text{ } 0\leq t \leq T) .
% \end{align*}
% Note that $M$ is non-decreasing and satisfies the assumptions of Theorem \ref{falling}. Thus, it is sufficient to find a lower bound for the one-sided exit problem with the boundary $M$.
% 
% For Theorem \ref{mblp} we have
% \begin{align*}
%   \PRO (X(t) \leq 1+ f(t), \text{ } 0\leq t \leq T) \geq  \PRO (X(t) \leq 1+ M(t), \text{ } 0\leq t \leq T) .
% \end{align*}
% Note that $M$ is non-decreasing and satisfies the assumptions of Theorem \ref{mblp}. In this case, it is sufficient to find an upper bound for the one-sided exit problem with the function $M$.
% \end{rem}

The following properties which are easy to check will be required for the proofs.
\begin{lemma}
Let $f : \mathbb{R}_+ \rightarrow \mathbb{R}$ be a non-decreasing function satisfying the assumptions of Theorem \ref{mblp}. Then,
\begin{align}\label{prop1}
 f(T) \leq c \cdot T, \text{ for all } T \text{ sufficiently large},
\end{align}
for some constant $c>0$. 
Furthermore, if the function $f$ satisfies additionally the assumptions of Theorem \ref{falling}, then there exists a constant $\tilde{c} >0$ such that
\begin{align}\label{prop2}
\sqrt{t} f'(s)  \leq \tilde{c} \quad \text{ a.e. for  all } s \geq t \geq1.
\end{align}
\end{lemma}
% \begin{proof}
%  The first property is a trivial observation since $\sup_{s\geq 1} |f'(s)| \leq d < \infty$. The second property follows from the fact that the function $f'(t) \searrow 0$ for $t\rightarrow \infty$ and $\int_1^{\infty} f'(s)^2 ds < \infty$. 
% \end{proof}

%%%%%%%%%%%%%%%%%%%%%%%%%%%%%%%%%%%%%%%%%%%%%%%%%%%%%%%%%%%%%%%%%%%%%%%%%%%%%%%%%%%%%%%%%%%%%%%%%%%%%%%%%%%%%%%%%%%%%%%%%%%%%%%%%%%%%%%%%%%%%%%%%%%%%%%%%%%%%%%%%%%%%%%%%%%%%%%%%%%%%%%%%%%%%5
For the proofs we use the Girsanov transform for additive processes to transform L\'evy processes into additive processes.
Let us recall that $N$ is a Poisson random measure on $(\mathbb{R}, \mathbb{R}^+)$ with intensity $\Lambda(dx,ds)$. The compensated measure is denoted by $\bar{N} (dx,ds) = N(dx,ds) - \Lambda (dx,ds)$. Furthermore, let $\PRO_X$ be a probability measure on $(D, \mathcal{F}_{D})$ where $D$ is the space of mappings from $[0,\infty)$ into $\mathbb{R}$ right continuous with left limits and $\mathcal{F}_{D}$ is the smallest $\sigma$-algebra that makes $X(t)$, $t\geq 0$, measurable (cf. \cite{sato}). 

The following theorem needed in the main proofs can be found in \cite{JS} (Theorem 3.24) and \cite{sato} (Theorems 33.1 and 33.2).
\begin{thm}\label{gir}
 Let $X$ and $Y$ be two additive processes with  triplets
$(\sigma_X^2, f_X (t), \Lambda_X (dx,dt))$ and $(\sigma_Y^2 , f_Y (t), \Lambda_Y (dx,dt))$, where  $\Lambda_X , \Lambda_Y$ are measures concentrated on $\mathbb{R}\backslash\{0\} \times [0,T]$. Then $\PRO_X |_{\FF_T}$ and $\PRO_Y |_{\FF_T}$  are absolutely continuous if and only if $\sigma_X = \sigma_Y$ and there exists $\theta(\cdot,\cdot) : \mathbb{R} \times [0,T] \rightarrow \mathbb{R}$ such that
\begin{itemize}
 \item $\int_0^T \int_{\mathbb{R} } \left( e^{\theta (x,s)/2} -1  \right)^2 \Lambda_X (dx,ds) < \infty,$ 
\item $\Lambda_X$ and $\Lambda_Y$ are  absolutely continuous with  $\frac{d \Lambda_Y}{d \Lambda_X} (x,s) = e^{\theta (x,s)}$, and
\item $\ f_Y (t)  =  f_X (t)  + \int_0^t \int_{|x| \leq 1} \left( e^{\theta (x,s)} -1  \right) x  \Lambda_X (dx,ds) $, for all $t \in [0,T]$.
\end{itemize}
The density transformation formula is given by
\begin{align}\label{dens}
 \frac{d \PRO_{Y}|_{\FF_T }}{d \PRO_X |_{\FF_T}} (X(\cdot)) = \exp & \left( - \int_0^T \int_{\mathbb{R} }  \left( e^{\theta(x,s)} -1- \theta(x,s) \right) \Lambda_X (dx,ds) + \right. \notag \\
& \left. \int_0^T \int_{\mathbb{R} }  \theta (x,s) \bar N_X (dx,ds) (\cdot)  \right) \quad \PRO_X \text{-a.s.}
\end{align}

\end{thm}

\begin{rem}
 The density transformation formula can also be expressed by
\begin{align}\label{dens2}
 \frac{d \PRO_{X}|_{\FF_T }}{d \PRO_Y |_{\FF_T}} (Y(\cdot)) = \exp & \left(  \int_0^T \int_{\mathbb{R} }  \left( e^{\theta(x,s)} -1- \theta(x,s) e^{\theta(x,s)} \right) \Lambda_X (dx,ds)  \right. \notag \\
& \left. - \int_0^T \int_{\mathbb{R}}  \theta (x,s) \bar N_Y (dx,ds) (\cdot)  \right) \quad \PRO_Y \text{-a.s.}
\end{align}
\end{rem}

%%%%%%%%%%%%%%%%%%%%%%%%%%%%%%%%%%%%%%%%%%%%%%%%%%%%%%%%%%%%%%%%%%%%%%%%%%%%%%%%%%%%%%%%%%%%%%%%%%%%%%%%%%%%%%%%%%%%%%%%%%%%%%%%%%%%%%%%%%%%%%%%%%%%%%%%%%%%%%%%%%%%%%%%%%%%%%%%%%%%

\subsection{One-sided exit problem with a moving boundary for Brownian motion}

%%%%%%%%%%%%%%%%%%%%%%%%%%%%%%%%%%%%%%%%%%%%%%%%%%%%%%%%%%%%%%%%%%%%%%%%%%%%%%%%%%%%%%%%%%%%%%%%%%%%%%%%%%%%%%%%%%%%%%%%%%%%%%%%%%%%%%%%%%%%%%%%%%%%%%%%%%%%%%%%%%%%%%%%%%5
Below, we present a lemma which deals with the one-sided exit problem for Brownian motion including a special kind of boundaries needed in the main proofs.
\begin{lemma}\label{BBgr}
Let $T>1$ and $c>0$ be a constant. Let $(B(t))_{t \geq0}$ be a Brownian motion. Define the function
\begin{align*}
 h_{T} (t)  := \max \left\{ (\ln T)^5 , t^{3/4} \right\}
\end{align*}
and the event
\begin{align*}
 E:= \left\{   B(t) \leq c \cdot h_T(t), \quad t \in [0,T]   \right\}.
\end{align*}
Then, we have 
\begin{align*}
 \pr{}{E^c} \lesssim e^{-(\ln T)^2 /4 }, \quad \text{ as } T \rightarrow \infty.
\end{align*}
\end{lemma}

\begin{proof}
First, note that $h_T (t) \geq g_T (t) := (\ln T) t^{6/10}$ for $t \geq 0$.

Define the event $\tilde{E}$  by
\begin{align*}
\tilde{E}:=  \left\{  B(t) \geq c \cdot g_T (t), \quad t \in[0,T] \right\} .
\end{align*}
Furthermore, denote by $\Phi$ the standard normal distribution function. Applying Theorem 4 and Example 7 in  \cite{jenler} it follows that
\begin{align*}
   \PRO \left(E^c\right)  \leq  \pr{}{\tilde{E}^c} &\lesssim 4 \left( \Phi \left((\ln T) T^{\tfrac{1}{10}} \right) - \Phi (\ln T)   \right) \leq \frac{\sqrt{2}}{\sqrt{ \pi}} e^{- (\ln T)^2 /4},
\end{align*}
for $T$ sufficiently large, which completes the proof.
\end{proof}

%%%%%%%%%%%%%%%%%%%%%%%%%%%%%%%%%%%%%%%%%%%%%%%%%%%%%%%%%%%%%%%%%%%%%%%%%%%%%%%%%%%%%%%%%%%%%%%%%%%%%%%%%%%%%%%%%%%%%%%%%%%%%%%%%%%%%%%%%%%%%%%%%%%%%%%%%%%%%%%%%%%%%%%%%%%%%%%%
\subsection{One-sided exit problem for L\'evy processes}
First, we study the asymptotic behaviour of the first passage time over a constant boundary. If Spitzer's condition holds, then \cite{GreNov}, Lemma 2, proves a similar result. 
\begin{lemma}\label{const1c}
  Let $X$ be a L\'evy process with L\'evy triplet $(\sigma^2 ,b,\nu)$. 
Let $\delta \geq 0$, $0 \leq a < T$ and $0 < c<\infty$. We have 
\begin{align*}
 \PRO (X(t) \leq 1, \text{ } a\leq t \leq T) = T^{-\delta+ o(1) }
\end{align*}
if and only if
\begin{align*}
 \PRO (X(t) \leq c, \text{ } a\leq t \leq T) = T^{-\delta+ o(1)}.
\end{align*}
\end{lemma}
\begin{proof} 
\textit{Case 1:} Let $c>1$. On one hand, we have
\begin{align*}
 \PRO (X(t) \leq 1, \text{ } a\leq t \leq T) \leq   \PRO (X(t) \leq c, \text{ } a  \leq t \leq T).
\end{align*}
On the other hand, let  $2 \leq  \lceil c \rceil := n \in \mathbb{N}$. Then, 
\begin{align*}
p_c (T) :=  \PRO (X(t) \leq c, \text{ } a\leq t \leq T) \leq   \PRO (X(t) \leq n, \text{ } a \leq t \leq T).
\end{align*}
Define $\tau_n := \inf \{ t \geq a: X(t) >n \}$ and let $F_{\tau_{n-1}}$ be the associated distribution function. The stationary and independent increments imply, for every $n \geq 2$, 
\begin{align*}
 p_n (T) & \leq p_{n-1} (T) + \int_a^T p_1 (T-s) dF_{\tau_{n-1}} (s)  \\
&\leq p_{n-1} (T) + p_1 (T/2) \int_a^{T/2} dF_{\tau_{n-1}} (s) + \int_{T/2}^T dF_{\tau_{n-1}} (s) \leq 3 p_{n-1} (T/2).
\end{align*}
 Thus,
\begin{align*}
 p_c (T) \leq p_n (T) \leq 3^{n-1} p_1 (T/2^{n-1}).
\end{align*}

\textit{Case 2:} Now, let $0<c<1$.  Then, on one hand, we have
\begin{align*}
  \PRO (X(t) \leq c, \text{ } a\leq t \leq T) \leq  \PRO (X(t) \leq 1, \text{ } a\leq t \leq T),
\end{align*}
and, on the other hand, analogously to Case 1 we obtain that
\begin{align*}
 p_1 (T) &= \PRO \left(\tfrac{1}{c} X(t) \leq \tfrac{1}{c} , \text{ } a \leq t \leq T  \right)  \leq d_1   \PRO \left(\tfrac{1}{c} X(t) \leq 1, \text{ } a \leq t \leq d_2  T  \right) = d_1  p_c (d_2  T),
\end{align*}
where $d_1, d_2 >0$ are dependent of $c$; and the lemma is proved.
\end{proof}

%%%%%%%%%%%%%%%%%%%%%%%%%%%%%%%%%%%%%%%%%%%%%%%%%%%%%%%%%%%%%%%%%%%%%%%%%%%%%%%%%%%%%%%%%%%%%%%%%%%%%%%%%%%%%%%%%%%%%%%%%%%%%%%%%%%%%%%%%%%%%%%%%%%%%%%%%%%%%%%%%%%%%%%%%%%%%%%%%%%%%%%%%%%%%%5
The following theorem provides a technique to decouple the one-sided boundary problem over different intervals.
\begin{lemma}\label{associated}
 Let $X$ be a L\'evy process with triplet $(\sigma^2, b, \nu)$ and $f: \mathbb{R}_+ \rightarrow \mathbb{R}$ be a measurable function. Let $0\leq a<b<c$. Then,
\begin{align*}
 \PRO &\left(X(t) \leq f(t), \text{ } a \leq t \leq c \right)  \geq \PRO \left(X(t) \leq f(t), \text{ } a \leq t \leq b \right) \cdot \PRO \left(X(t) \leq f(t), \text{ } b \leq t \leq c \right).
\end{align*}
\end{lemma}
\begin{proof}
 For any choice of $n$ and $0\leq t_1 < ...< t_n$ the random variables $(X(t_i))^n_{i=1}$ are associated (cf.\ \cite{ass}), since they are sums of independent random variables. Hence, the functions $\mathbf{1}_{\{ X(t) \leq f(t), \text{ } a\leq t \leq b\} }$ and
$\mathbf{1}_{\{X(t) \leq f(t), \text{ } b\leq t \leq c \}}$ can both be written as limits of decreasing functions of associated random variables and are thus also associated. Hence, we obtain the desired assertion.
\end{proof}

%%%%%%%%%%%%%%%%%%%%%%%%%%%%%%%%%%%%%%%%%%%%%%%%%%%%%%%%%%%%%%%%%%%%%%%%%%%%%%%%%%%%%%%%%%%%%%%%%%%%%%%%%%%%%%%%%%%%%%%%%%%%%%%%%%%%%%%%%%%%%%%%%%%%%%%%%%%%%
Furthermore, we need a result for one-sided exit problem with a  boundary that is an increasing function of $T$.
\begin{lemma}\label{helpln}
 Let $X$ be a L\'evy process with L\'evy triplet $(\sigma^2,b,\nu)$. Then we have, for $T$ sufficiently large,
\begin{align*}
  \PRO & ( X(t) \leq 3, \text{ }0 \leq t \leq T) \\
& \geq \tfrac{1}{2} \PRO \left( X(t) \leq 3 - t^{1/3}, \text{  } 0 \leq t \leq  ( \ln T)^{21}  \right) \cdot \PRO \left(X (t) \leq 3 + (\ln T)^6, \text{  } 1 \leq t \leq T \right). 
\end{align*}

\end{lemma}
\begin{proof}
Note that that $(\ln T)^7 \geq 3 + (\ln T)^6$, for $T$ sufficiently large, and due to the stationary and independent increments of $(X (t) )_{t \geq 0}$ we have, for $T $ sufficiently large, 
\begin{align*}
\PRO & \left(   X (t) \leq 3 -t^{1/3} , \text{  } 0 \leq t \leq (\ln T)^{21}  \right)  \cdot \PRO \left(   X(t)  \leq 3 + (\ln T)^6, \text{  }    0 \leq t \leq T   \right) \\
& \leq \PRO \left(   X (t) \leq 3 -t^{1/3} , \text{  } 0 \leq t \leq (\ln T)^{21}  \right)\\
& \quad \cdot \PRO \Bigl(   X(t) - X ((\ln T)^{21}) \leq 3 + (\ln T)^6, \text{  }  ( \ln T)^{21}  \leq t \leq  T  \Bigr) \\
& \leq \PRO \left(  \{  X (t) \leq 3 -t^{1/3} , \text{  } 0 \leq t \leq (\ln T)^{21} \}  \cap \{ X(t) \leq 3 , \text{  }  ( \ln T)^{21}  \leq t \leq T\}  \right) \\
& \leq  \PRO  (X (t) \leq 3 , \text{  } 0 \leq t \leq T).
\end{align*} 
Lemma \ref{associated} yields
\begin{align*}
 \PRO \left(   X(t)  \leq 3 + (\ln T)^6, \text{  }    0 \leq t \leq T   \right) \geq \tfrac{1}{2}  \PRO \left(   X(t)  \leq 3 + (\ln T)^6, \text{  }    1 \leq t \leq T   \right),
\end{align*}
since $\PRO (X(t)  \leq 3 + (\ln T)^6, \text{ } 0\leq t \leq 1)>\tfrac{1}{2}$, for $T$ sufficiently large.
\end{proof}

%%%%%%%%%%%%%%%%%%%%%%%%%%%%%%%%%%%%%%%%%%%%%%%%%%%%%%%%%%%%%%%%%%%%%%%%%%%%%%%%%%%%%%%%%%%%%%%%%%%%%%%%%%%%%%%%%%%%%%%%%%%%%%%%%%%%%%%%%
 Here, we show that, if the boundary is equal to $t^{\alpha}$, $\alpha>1/2$ then the probability of the one-sided exit problem for a L\'evy martingale with $\ER \left( |X(1)|^q \right) < \infty$, for some $q >4$, over the boundary $t^{\alpha}$ is larger than a constant.
\begin{lemma}\label{coup}
Let $X$ be a L\'evy martingale with  $\ER \left( |X(1)|^q \right) < \infty$, for some $q >4$. Then, for any $\alpha >1/2$,
\begin{align*}
\PRO \left(X (t) \leq t^{\alpha}, \text{ } 1 \leq t \leq T \right) \gtrsim c, \quad \text{ as } T \rightarrow \infty,
\end{align*}
where $c>0$ is a constant depending only on $X$ and $\alpha$.
\end{lemma}
\begin{proof}
First note that there exists  $\eps>0$ such that $q> 2(1+\eps) +2$. Since $\alpha>1/2$ there exists $\beta >0$ such that $\alpha -\beta  -\tfrac{1}{2} >0$. Choose natural number $K := K(X,\alpha,\beta) >0$ independent of $T$ such that $K \geq 2^{1/\beta}$ and 
\begin{align}\label{sum}
 \sum_{n=K}^{\infty}   n^{-(1+\eps)}    \leq \frac{1}{2} \left[ \frac{\sqrt{2}}{3 \sqrt{\pi}} + 2^{-(1+\eps)/\alpha}  \ER \left(  |X(1)|^{(1+\eps)/\alpha} \right)  \right]^{-1} .
\end{align}
 Then,  Lemma \ref{associated} yields for every $T>K$
\begin{align}\label{slep1}
 g(T) &:=  \PRO \left(  X(t) \leq t^{\alpha} , \text{ } 1\leq t\leq T \right) \notag \\
&\geq g(K) \cdot \left( 1- \PRO \left( \exists \text{ } t \in [K,T]: X(t) > t^{\alpha}\right)  \right)  \notag  \\
&\geq g(K) \cdot \left( 1- \sum_{n=K}^{\lfloor T \rfloor} \PRO \left( \exists \text{ } t \in (n,n+1]: X(t) > t^{\alpha}\right)  \right). 
\end{align}
On the other hand, due to the stationary and independent increments we obtain, for all $n \geq K$,
\begin{align}\label{alln1}
 \PRO & \left( \exists \text{ } s \in (n,n+1]: X (s) > s^{\alpha}\right) \notag  \\
& \leq \PRO \left( X(n) \geq n^{\alpha -\beta} \right) + \PRO \left( \{X(n) < n^{\alpha -\beta} \} \cap \{ \exists \text{ } s \in (n,n+1]: X(s) > s^{\alpha} \}\right) \notag  \\
& \leq \PRO \left( X(n)/\sqrt{n} \geq n^{\alpha -\beta - 1/2} \right) + \PRO  \left(\exists \text{ } s \in (n,n+1]: X(s) -X(n) > s^{\alpha} - n^{\alpha -\beta} \right) \notag  \\
& \leq  \PRO \left( X(n)/\sqrt{n} \geq 3 \sqrt{\ln n} \right)      + \PRO  \left(\exists \text{ } s \in (n,n+1]: X(s) -X(n) > \tfrac12 n^{\alpha} \right) \notag  \\  
& \leq  \frac{\sqrt{2}}{3 \sqrt{\pi \ln n}}\cdot  n^{-(1+\eps)}  + \PRO  \left( \exists \text{ }s \in (0,1]: |X(s)| > \tfrac12 n^{\alpha} \right)\notag  \\ 
& \leq  \frac{\sqrt{2}}{3 \sqrt{\pi}}\cdot n^{-(1+\eps)}  + 2^{-(1+\eps)/\alpha}  \ER\left(   |X(1)|^{(1+\eps)/\alpha} \right) \cdot n^{-(1+\eps)}  ,
\end{align}
where we used in the second last step  a result of \cite{petrov}, page 254, and in the last step Doob's martingale inequality. Putting (\ref{alln1}) and (\ref{sum}) into (\ref{slep1}) yields
\begin{align*}
 g(T) \geq g(K)/2 >0,
\end{align*}
which proves the lemma. 
\end{proof}

%%%%%%%%%%%%%%%%%%%%%%%%%%%%%%%%%%%%%%%%%%%%%%%%%%%%%%%%%%%%%%%%%%%%%%%%%%%%%%%%%%%%%%%%%%%%%%%%%%%%%%%%%%%%%%%%%%%%%%%%%%%%%%%%%%%%%%%%%%%%%%%%%%%%%%%%%%%%%%%%%%%%%%%%%%%%%%%%%%%%%%%%%%
\subsection{Coupling}

With the help of a coupling method we also obtain an upper bound for the one-sided exit problem for a L\'evy martingale with some finite exponential moment. 
%%%%%%%%%%%%%%%%%%%%%%%%%%%%%%%%%%%%%%%%%%%%%%%%%%%%%%%%%%%%%%%%%%%%%%%%%%%%%%%%%%%%%%%%%%%%%%%%%%%%%%%%%%%%%%%%%%%%%%%%%%%%%%%%%%%%%%%%%%%%%%%%%%%%%%%%%%%%%%%%%%%5
\begin{lemma}\label{coup2}
Let $c>0$. Let $X_1$ and $X_2$ be two independent L\'evy processes, where $X_2$ is a martingale with some finite exponential moment, i.e. $\ER \left( e^{b|X_2(1)|} \right) < \infty$, for some $b >0$. Furthermore, let $\ER \left( X_2  (1)^2 \right)  =a$. Let $B$ be a Brownian motion and $f : \mathbb{R}_+ \rightarrow \mathbb{R}_+$ be a non-decreasing function such that there exists a constant $d>0$ with $f(T) \leq d \cdot T$, for $T$ sufficiently large.
Then there is a $\kappa_c>0$ depending on $c$ such that, for $T$ sufficiently large, 
\begin{align*}
 &\PRO \Bigl( X_1 (t) + X_2 (f(t)) \leq 1 , \text{ } 1 \leq t \leq T \Bigr) \\
&\leq  \PRO \Bigl( X_1 (t) + aB(f(t)) \leq 1 + \kappa_c \ln T, \text{ } 1 \leq t \leq T \Bigr) + T^{-c }.
\end{align*}
\end{lemma}

\begin{proof}
 Since $X_2$ has some finite exponential moment and $\ER X_2 (1)^2 =a$, one can couple it with a Brownian motion $aB$ (compare to the Koml\'os-Major-Tusn\'ady coupling (KMT theorem), \cite{KMT}) in such a way that, for a suitable $\kappa_c >0$ and $T$ sufficiently large,
\begin{align*}
 \PRO \left(  \sup_{0 \leq t\leq T} |X_2 (t) - aB (t) |> \frac{\kappa_c}{2} \ln T \right) \leq T^{-c}.
\end{align*}
Since $f(T)\leq d \cdot T$, for $T$ sufficiently large, we have
\begin{align}\label{coupl2}
 & \PRO \left(  \sup_{1 \leq t\leq T} |X_2 (f(t)) - aB (f(t)) |>  \kappa_c \ln T \right) \notag \\
& \quad \leq  \PRO \left(  \sup_{0 \leq t\leq d T} |X_2 (t) - aB (t) |> \kappa_c \ln T \right) \notag \\
& \quad \leq \PRO \left(  \sup_{0 \leq t\leq  \max\{T,dT\} } |X_2 (t) - aB (t) |> \frac{\kappa_c}{2} \ln (\max\{T,dT\} ) \right) \notag \\
&\quad  \leq  \min\{1, d^{-c}\} T^{-c} \leq T^{-c}.
\end{align}

Define 
\begin{align*}
 A:= \left\{ \sup_{ 1 \leq t \leq T}  |X_2 (f(t)) - aB (f(t)) |\leq  \kappa_c \ln T \right\}
\end{align*}
to be the set where the coupling works. Then, by inequality  (\ref{coupl2}), for $T$ sufficiently large,
\begin{align*}
 & \PRO \Bigl( X_1 (t) + X_2 (f(t)) \leq 1 , \text{ } 1 \leq t \leq T \Bigr) \\
 & \leq  \PRO \Bigl( X_1 (t) + X_2 (f(t)) \leq 1 , \text{ } 1 \leq t \leq T; A \Bigr) + \PRO \Bigl( A^c \Bigr)\\
& \leq \PRO \Bigl( X_1 (t) + aB (f(t)) \leq 1 + \kappa_c \ln T, \text{ } 1 \leq t \leq T \Bigr) +  T^{-c},
\end{align*}
which completes the proof.
\end{proof}

%%%%%%%%%%%%%%%%%%%%%%%%%%%%%%%%%%%%%%%%%%%%%%%%%%%%%%%%%%%%%%%%%%%%%%%%%%%%%%%%%%%%%%%%%%%%%%%%%%%%%%%%%%%%%%%%%%%%%%%%%%%%%%%%%%%%%%%%%%%%%%%%%%%%%%%%%%%%%%%%%%%%%%%%%%%%

%%%%%%%%%%%%%%%%%%%%%%%%%%%%%%%%%%%%%%%%%%%%%%%%%%%%%%%%%%%%%%%%%%%%%%%%%%%%%%%%%%%%%%%%%%%%%%%%%%%%%%%%%%%%%%%%%%%%%%%%%%%%%%%%%%%%%%%%%%%%%%%%%%%%%%%%%%%%%%%%%%%%%%%%%%%%%%%%%%%%%%%%%%%%%%%%%%%%%%%%%%%%%%%%%%%%%%%%%%%%%%%%

%%%%%%%%%%%%%%%%%%%%%%%%%%%%%%%%%%%%%%%%%%%%%%%%%%%%%%%%%%%%%%%%%%%%%%%%%%%%%%%%%%%%%%%%%%%%%%%%%%%%%%%%%%%%%%%%%%%%%%%%%%%%%%%%%%%%%%%%%%%%%%%%%%%%%%%%%%%%%%%%%%%%%%%%%%%%%%%%%
\section{Proof of Theorem \ref{falling} (negative boundary) }\label{proofthm2}
 Since $f(t)$ is positive, our quantity is trivially bounded from above as follows
\begin{align*}
  \PRO (X(t) \leq 1- f(t), \text{ } 0\leq t \leq T) \leq \PRO (X(t) \leq 1, \text{ } 0\leq t \leq T) =  T^{-\delta + o(1)}.
\end{align*}
In order to prove the  lower bound we  can assume that $T>1$ during the further progress of the proof and introduce the auxiliary functions $H^i_\beta$ and $f_n$. 

We define
\begin{align*}
 H(x) := x \exp\left(   -\sqrt{c_1 ||f'||^2_{L_2 [1,\infty)} \ln (1 / x )} - c_2 ||f'||^2_{L_2 [1,\infty)} \right), \text{ for  } x  \in (0,1],
\end{align*}
where $c_1,c_2 >0$ are constants depending on $\nu$ and $f$ specified later.
Note that $H'(x) >0$ on $(0,1]$.
 Next, define $H^{i}_\beta $ by $H^0_\beta (x) := x$ and, for $i \geq 1$, 
\begin{align*}
 H^i_\beta (x) := H^{i-1}_\beta \Bigl( H(x \cdot \beta) \Bigr)
\end{align*}
with $0<\beta <1$ specified later. Note that $H^i_\beta $ is well defined since  $H(x)\in (0,1]$ for $x \in (0,1]$.\\

Next, we define  $f_0 (t) := \max \{ f(\ln T),f(t) \}$ and, for $n \geq 1$,
\begin{align*}
   f_n (t) :=  \max \left\{ 1,  \left( f_{n-1} (t)- f_{n-1} (\ln T)\right)^{2/3} \right\} + f_{n-1} (\ln T), \quad t \geq 0.
\end{align*}
Furthermore, define $\tilde{t}_n := \sup \{s\geq 0: f_{n-1} (s)-f_{n-1} (\ln T) \leq 1 \}$. Note that $f'_n(t)=0$, for $t\in (0,\tilde{t}_n)$, and
\begin{align*}
 f'_n (t) =  \tfrac{2}{3}  \left( f_{n-1}(t) - f_{n-1} (\ln T) \right)^{-2/3}  f_{n-1}'(t) \quad  \text{ a.e.,  for }  t>\tilde{t}_n,
\end{align*}
thus, 
\begin{align}\label{incr2}
 0 \leq f'_{n} (t) \leq f' (t) \text{ a.e., }
\end{align}
since $f' \geq 0$.
In the following proof we use 
\begin{align}\label{findu}
 f_n(t) \leq f(\ln T) + n + \max \{  1, f(t)^{(2/3)^n}  \}, \quad \text{ for all } t \geq 0, 
\end{align}
which can be proved by induction.

We proceed with the proof of the lower bound which includes two iterations.

\subsection{External iteration}
In this section we provide an iteration method in order to apply the results of Section \ref{initer}. This additional step is required because of technical details in Section \ref{initer} which contains the main idea of this proof. Therefore, define, for any $T>1$,
\begin{align*}
 G(T) :=  \PRO (X(t) \leq 1- f(t), \text{ } \ln T \leq t \leq T). 
\end{align*}
In Section \ref{initer} we will prove that 
\begin{align} \label{subB}
 G(T) \geq T^{-\delta+o(1)} \cdot G(\ln T), \quad \text{ for all } T>1.
\end{align}
Recall that $\ln^* (T)$ is the number of times the logarithm function must be iteratively applied before the result is less than or equal to one.  Denote $\ln^n (T) $  the $n$-times iteratively applied logarithm and $\ln^0 (T) := T$. Moreover, note that $\ln^* (T)$ decays slower than $\ln^k (T)$, for every $k$.

Lemma \ref{associated} yields
\begin{align*}
 \PRO & (X(t) \leq 1- f(t), \text{ } 0 \leq t \leq T) \\
&\geq \PRO (X(t) \leq 1- f(t), \text{ } 0 \leq t \leq \ln^{\ln^* (T)} ) \cdot G(\ln^{\ln^*(T)-1}(T)) \cdot ...\cdot G(\ln T) \cdot G(T)\\
&\geq G(1) \cdot G(\ln^{\ln^*(T)-1}(T)) \cdot ...\cdot G(\ln T) \cdot G(T)\\
&= G(1) \prod_{k=0}^{\ln^*( T)-1} G( \ln^k (T)).
 \end{align*}
Combining this with (\ref{subB}) and the fact that $\ln^j (T) \leq \ln^k (T)$, for all $j \geq k \geq 0$, which will be used in the third and fourth  step, and $\ln^* (T) \leq \ln^3 (T)$, for $T$ sufficiently large, we obtain that 
\begin{align*}
 \PRO& \Bigl(X(t) \leq 1- f(t), \text{ } 0 \leq t \leq T \Bigr) \\
&\geq  G(1) \left( \prod_{k=1}^{\ln^*( T)-1} G( \ln^{k} (T))  \right) \cdot G(\ln T) \cdot T^{-\delta +o(1)} \\
&\geq   G(1)^{\ln^* (T)+1}  \left( \prod_{k=1}^{\ln^*( T)-1}  \prod_{j=k}^{\ln^* (T) -1}  \left( \ln^j (T)\right)^{-\delta +o(1)} \right)   \cdot T^{-\delta +o(1)} \\
 &\geq G(1)^{\ln^* (T)+1} \left( \prod_{k=1}^{\ln^*( T)-1} \left(  (\ln^k (T))^{-\delta +o(1)}  \right)^{\ln^* (T) -k} \right) \cdot T^{-\delta +o(1)}\\
&\geq   G(1)^{\ln^* (T)+1}  \left(     \prod_{k=1}^{\ln^*( T)-1}  \left( \ln^1 (T)\right)^{-(\ln^* (T) -k) \delta +(\ln^* (T) -k)o(1)} \right) \cdot T^{-\delta +o(1)} \\
&\geq   G(1)^{\ln^* (T)+1}  \left(  \ln  T \right)^{-\ln^3 (T)} \cdot T^{-\delta +o(1)} \\
&= T^{-\delta +o(1)},
\end{align*}
 and this is precisely the assertion of the theorem.

\subsection{Internal iteration; proof of (\ref{subB}) }\label{initer}

First, define 
\begin{align*}
g_n (T) :=  \PRO (X(t) \leq 1 - f_n (t), \text{  } \ln T \leq t \leq T).
\end{align*}

% \subsubsection{Iteration rule}

% \paragraph{Goal setting}

%The following steps will be conducted for fixed $n \in \mathbb{N}$.

\paragraph*{Step 1: Proof of (\ref{mass2})} \label{change}
By using a change of measure the aim of this step is to show the following inequality 
\begin{align}\label{mass2}
g_n (T)% &= \PRO (\tilde{X}_n (t) \leq 1, \text{  } \ln T \leq t \leq T) \notag \\
&\geq \PRO (Y_n (t) \leq 1, \text{  } \ln T \leq t \leq T) \notag \\
 & \quad \cdot \exp \left(- \sqrt{c_1 ||f'||^2_{L_2 [1,\infty)} \ln (1/ \PRO (Y_n(t) \leq 1, \text{  } \ln T \leq t \leq T) )} - c_2   ||f'||^2_{L_2 [1,\infty)} \right) \notag\\
&= H \left(\PRO (Y_n (t) \leq 1, \text{  } \ln T \leq t \leq T) \right),
\end{align}
where $c_1,c_2 >0$ are constants depending on $\nu$ and $f$ that are chosen later on.

Without loss of generality let $\nu ([-1,0)) >0$. If $\nu ([-1,0)) = 0$ then we multiply $X$ by $d>0$ suitably chosen such that $\tilde{\nu} ([-1,0]))>0$, where $\tilde{\nu}$ is the L\'evy measure of $d\cdot X$. Such  $d>0$ exists since $\nu(\mathbb{R}_-)>0$. Due to Lemma \ref{const1c} we can continue with the process $d \cdot X$ instead of $X$ in the same manner.  

Since $\nu ([-1,0)) >0$ we can choose  a compact set $A \subseteq [-1 ,0)$ such that 
\begin{align*}
0<  \int_A x^2 \nu(dx)=: m < \infty.
\end{align*}

 Let $\tilde{X}_n$ and $Y_n$ be two additive processes with triplets $(\sigma^2, f_{\tilde{X}_n} (t), \nu(dx)ds )$ and  \\
$(\sigma^2, f_{Y_n} (t) , (1+ \frac{f_n'(t) |x|}{m} \mathbf{1}_{\{x \in A \}}) \nu(dx)ds )$ respectively, where $f_{Y_n} (t) := b\cdot t + f_n (\ln T)$ and $f_{\tilde{X}_n} (t) := b\cdot t + f_n (t)$. 

 Then,  $\PRO_{\tilde{X}_n} |_{\FF_T}$ and $\PRO_{Y_n} |_{\FF_T}$ are absolutely continuous because of the following considerations. Define $\theta(x,s) := \ln (1 + \tfrac{ f_n'(s)|x|}{m} \mathbf{1}_{\{x \in A\}})$, for all $ s \in [0,T]$ and  $x \in \mathbb{R}$. Using the fact that $f_n' (s) = 0 $for $s \in (0, \ln T)$ we have, for $t>\ln T$, 
\begin{align*}
 f_{Y_n} (t)  &=  bt+f_n(\ln T) = bt + f_n (t) - \int_{\ln T}^t f'_n (s) ds\\
% &=   f_{\tilde{X}_n} (t)  + \int_{\ln T}^t \int_A x^2 \tfrac{(-f'_n(s))}{m} \nu(dx) ds \\
&= f_{\tilde{X}_n} (t)  + \int_0^t \int_{|x|\leq 1 } \left( e^{\theta (x,s)} -1  \right) x  \nu(dx)ds
\end{align*}
and since $f_n (t) = f_n (\ln T)$, for $t\in [0,\ln T]$,
\begin{align*}
  f_{Y_n} (t)  = bt + f_n (\ln T) = bt +f_n (t)   = f_{\tilde{X}_n} (t) .
\end{align*}
In this connection, one should point out that $-f'_n (s) x \mathbf{1}_{\{x \in A\}} = f'_n (s) |x|  \mathbf{1}_{\{x \in A\}}\geq 0$ almost everywhere.

Define $\Lambda_{Y_n} (dx, ds) := \exp( \theta (x,s)) \nu (dx)ds$.  
According to the choice of the L\'evy measures, $\nu(dx)ds$ and $\Lambda_{Y_n} (dx,ds)$ are  absolutely continuous with  $\frac{d \Lambda_{Y_n} (x,s)}{\nu (dx) ds} = e^{\theta (x,s)}$. In order to apply Theorem \ref{gir} we have to check 
$\int_0^T \int_{\mathbb{R}} \left( e^{\theta (x,s)/2} -1  \right)^2 \nu (dx)ds < \infty$. We know from \cite{sato}, Remark 33.3, that this condition is equivalent to the following three properties combined
\begin{enumerate}
 \item $ \int_{\{(x,s) : \theta (x,s) <-1\}} \nu (dx)ds <\infty$, 
\item $\int_{\{(x,s) : \theta (x,s) >1\}} e^{\theta (x,s)} \nu (dx)ds <\infty$, and
\item $\int_{\{(x,s) : |\theta (x,s)| \leq 1\}}  \theta^2 (x,s) \nu (dx)ds <\infty$.
\end{enumerate}
Since $f_n' \geq 0$, thus $\theta \geq 0$;  it is left to prove 2. and 3.

\textit{Case 2.}: Since $\theta > 1$ and $A$ bounded away from zero, we have
\begin{align*}
\int_{\{(x,s) : \theta (x,s) >1\}} e^{\theta (x,s)} \nu (dx)ds \leq \int_1^T \int_{A}( 1+ \frac{f_n'(s) |x|}{m} )\nu(dx)ds < \infty.
\end{align*}

\textit{Case 3.}: Since $\ln (1+z) \leq z$, for all $z > -1$, and inequality (\ref{incr2}) we get
\begin{align*}
 \int_{\{(x,s) :| \theta (x,s)| \leq 1\}} (\theta (x,s))^2 \nu (dx)ds &\leq \frac{1}{m^2} \int_1^T \int_A  (f_n' (s))^2 x^2 \nu(dx)ds =  \frac{1}{m} ||f_n'||^2_{L_2 [1,T]} < \infty.
\end{align*}
Hence, due to Theorem \ref{gir} $\PRO_{\tilde{X}_n} |_{\FF_T}$ and $\PRO_{Y_n} |_{\FF_T}$ are absolutely continuous.

Next, we show inequality (\ref{mass2}).

Note that $\theta (x,s) = 0$, for $s \in [0,\ln T)$ and all $x \in \mathbb{R}$. Because of Theorem \ref{gir} and the density transformation formula (\ref{dens2}) we obtain that
\begin{align}\label{lala}
 \PRO & (\tilde{X}_n (t) \leq 1, \text{ } \ln T \leq t \leq T) = \ER_{\tilde{X}_n} ( \mathbf{1}_{\{\tilde{X}_n(t) \leq 1, \text{ } \ln T \leq  t \leq T\}}) \notag \\
%&=  \ER_{Y_n} \left( \mathbf{1}_{\{Y_n (t) \leq 1, \text{ } \ln T \leq  t \leq T\}}  \frac{d\PRO_{\tilde{X}_n}}{d\PRO _{Y_n}} (Y_n (\cdot)) \right) \notag \\
 &= \ER_{Y_n} \left( \mathbf{1}_{\{Y_n (t) \leq 1, \text{ } \ln T \leq  t \leq T\}} e^{ - \int_{\ln T}^T \int_{\mathbb{R}}  \theta (x,s)  \bar{N}_{Y_n} (dx,ds)} \right)  \cdot e^{ \int_{\ln T}^T \int_{\mathbb{R}}  \left( e^{\theta(x,s)} -1- \theta(x,s) e^{\theta(x,s)} \right) \nu(dx)ds }\notag \\
%  &= \ER_{Y_n} \left( \mathbf{1}_{\{Y_n (t) \leq 1, \text{ } \ln T \leq  t \leq T\}} e^{ - \int_{\ln T}^T \int_{\mathbb{R}}  \theta (x,s)  \bar{N}_{Y_n} (dx,ds)} \right) \notag\\
%  &\quad \cdot e^{ \int_{\ln T}^T \int_{\mathbb{R} }  \left(  \tfrac{f_n'(s) |x|}{m} \mathbf{1}_{x\in A }   - \ln \left( 1+\tfrac{f_n'(s) |x|}{m} \mathbf{1}_{x\in A }  \right)     \left(1+\tfrac{f_n'(s) |x|}{m} \mathbf{1}_{x\in A }  \right)       \right) \nu(dx)ds} \notag\\
&= \ER_{Y_n} \left( \mathbf{1}_{\{Y_n (t) \leq 1, \text{ } \ln T \leq  t \leq T\}} e^{ - \int_{\ln T}^T \int_{\mathbb{R}}  \theta (x,s)  \bar{N}_{Y_n} (dx,ds)} \right) \cdot e^{- \int_{\ln T}^T \int_{\mathbb{R} }  g \left(\tfrac{f_n'(s) |x|}{m} \mathbf{1}_{x\in A } \right) \nu (dx)ds },
\end{align}
where $g(u) := (1+u ) \ln (1+u) - u$, $u>0$. For $u\geq 0$ bounded away from infinity, we have with a constant $\tilde{c}_1>0$, $g(u) \leq \tilde{c}_1 u^2$ because of Taylor's expansion. Hence, since $f'_n$ and $A$ is bounded away from $-\infty$, we get
\begin{align*}
 & e^{- \int_{\ln T}^T \int_{\mathbb{R}}  g \left(\tfrac{f'_n (s)|x|}{m} \mathbf{1}_{x\in A } \right) \nu (dx)ds } \geq e^{- \tilde{c}_1 \int_{\ln T}^T \int_{\mathbb{R}}   \tfrac{f'_n (s)^2 x^2}{m^2} \mathbf{1}_{x\in A } \nu (dx)ds } \\
 &= e^{- \tilde{c}_1 \int_{\ln T}^T f'_n (s)^2 ds \cdot  \int_{A}   \tfrac{ x^2}{m^2} \nu (dx) } = e^{- \frac{\tilde{c}_1}{m} ||f'_n ||^2_{L_2[\ln T,T]}} \geq e^{- \frac{\tilde{c}_1}{m} ||f' ||^2_{L_2[1,\infty)}},
\end{align*}
having used (\ref{incr2}).
Let $p>1$. Using the last estimate and the reverse H\"older inequality in (\ref{lala}) yields that
\begin{align}\label{dens1}
 \PRO  &(\tilde{X}_n (t) \leq 1,  \text{ } \ln T \leq t \leq T) \notag \\
& \geq \exp \left(- \frac{\tilde{c}_1}{m} ||f'_n||^2_{L_2[1,\infty)}\right) (\PRO (Y_n (t) \leq 1,\text{ } \ln T \leq t\leq T))^p \notag \\
& \quad \cdot \left( \ER_{Y_n} \left(e^{\frac{1}{p-1} \int_{\ln T}^T \int_{\mathbb{R} }  \theta (x,s) \bar{N}_{Y_n} (dx,ds)  }  \right)\right)^{-(p-1)} .
\end{align}
Furthermore, we have due to the density transform formula (\ref{dens})
 \begin{align*}
 & \left(  \ER_{Y_n} \left(e^{\frac{1}{p-1} \int_{\ln T}^T \int_{\mathbb{R}}  \theta (x,s) \bar{N}_{Y_n} (dx,ds)  }  \right)\right)^{-(p-1)} \\
&  = \left( \ER_{\tilde{X}_n} \left(e^{ \int_{\ln T}^T \int_{\mathbb{R} } \frac{1}{p-1} \theta (x,s) ( N (dx,ds) - \Lambda_{Y_n} (dx,ds)  ) +  \theta (x,s) ( N (dx,ds) - \nu (dx)ds ) }  \right)\right)^{-(p-1)}\\
 & \quad \cdot    \left(e^{ - \int_1^T \int_{\mathbb{R} }  (e^{\theta (x,s)} - 1- \theta(x,s))  \nu (dx)ds }  \right)^{-(p-1)} \\
&=  \left( \ER_{\tilde{X}_n} \left(e^{ \int_{\ln T}^T \int_{\mathbb{R} } (\frac{1}{p-1}+1)  \theta (x,s) ( N (dx,ds) - \nu (dx)ds) }  \right)\right)^{-(p-1)}\\
 & \quad \cdot   \left( \exp{ \left( \int_{\ln T}^T \int_{\mathbb{R} }  (\frac{\theta(x,s)}{p-1}- \frac{\theta(x,s)}{p-1} e^{\theta(x,s)}-e^{\theta (x,s)} + 1+ \theta(x,s))  \nu (dx)ds \right) }  \right)^{-(p-1)} \\
&=   \left(\exp{ \left(  \int_{\ln T}^T \int_{\mathbb{R} }  (e^{(\tfrac{1}{p-1} +1)\theta (x,s)} -1 - (\tfrac{1}{p-1} +1) \theta(x,s))  \nu (dx)ds \right) }  \right)^{-(p-1)} \\
 & \quad \cdot    \left(\exp{ \left(  \int_{\ln T}^T \int_{\mathbb{R} }  (\frac{\theta(x,s)}{p-1}- \frac{\theta(x,s)}{p-1} e^{\theta(x,s)}-e^{\theta (x,s)} + 1+ \theta(x,s))  \nu (dx)ds \right) }  \right)^{-(p-1)} \\
% &= \left(\exp{ \left( \int_{\ln T}^T \int_{\mathbb{R} } e^{\theta(x,s)} (e^{\tfrac{1}{p-1}\theta (x,s)} -1 - \tfrac{1}{p-1} \theta(x,s))  \nu (dx)ds \right) }  \right)^{-(p-1)}   \\
&= \exp{\left( (p-1)  \int_{\ln T}^T \int_{\mathbb{R} } e^{\theta(x,s)} (-e^{\tfrac{1}{p-1}\theta (x,s)} +1 + \tfrac{1}{p-1} \theta(x,s))  \nu (dx)ds \right) }  ,
 \end{align*}
where we used  in the third step a modification of Lemma 33.6 of \cite{sato}. The difference between \cite{sato} and our case consists in the consideration of time-inhomogeneous processes in contrast to time-homogeneous processes used in \cite{sato}. More precisely, we apply this Lemma to the following process
\begin{align*}
 \int_{\ln T}^T \int_{\mathbb{R} } & (\frac{1}{p-1}+1)  \theta (x,s) ( N (dx,ds) - \nu (dx)ds) \\
&-\int_{\ln T}^T \int_{\mathbb{R} }  (e^{(\tfrac{1}{p-1} +1)\theta (x,s)} -1 - (\tfrac{1}{p-1} +1) \theta(x,s))  \nu (dx)ds,
\end{align*}
and use the properties of the Girsanov transform for additive processes (Theorem \ref{gir}) instead for L\'evy processes.  Next, define $w(x) := 1+x-e^x$, for all $x \geq 0$. Assume for a moment that $p>1$ is chosen such that 
\begin{align}\label{bound}
 \tfrac{1}{p-1} \theta(x,s), \quad  \text{ for all }x\in \mathbb{R} \text{ and } s \in [\ln T, T],
\end{align}
is almost everywhere bounded away from infinity. This boundedness is independent of $T$ and $n$.
 Then, there is a constant $\tilde{c}_2 >0$ such that $w(\tfrac{1}{p-1} \theta(x,s)) \geq - \tilde{c}_2(\tfrac{1}{p-1} \theta(x,s))^2$ and hence, 
\begin{align*}
 & (p-1)  \int_{\ln T}^T \int_{\mathbb{R} } e^{\theta(x,s)} (-e^{\tfrac{1}{p-1}\theta (x,s)} +1 + \tfrac{1}{p-1} \theta(x,s)) \nu (dx) ds   \\
& \geq - \frac{\tilde{c}_2}{(p-1)}  \int_{\ln T}^T  \int_{\mathbb{R}} (\theta (x,s))^2 \nu(dx) ds \geq - \frac{\tilde{c}_2}{(p-1)m^2}  \int_{\ln T}^T  (f'_n (s))^2 ds \cdot \int_{A} x^2  \nu (dx)  \\
&\geq - \frac{\tilde{c}_2}{(p-1)m} || f'||^2_{L_2 [1,\infty)},
\end{align*}
where we used in the last step again inequality (\ref{incr2}).
Putting this into (\ref{dens1}) implies
\begin{align*}
 \PRO & (\tilde{X}_n (t) \leq 1, \text{ } \ln T \leq  t \leq T) \\
& \geq \PRO  (Y_n (t) \leq 1, \text{ } \ln T \leq t \leq T)^{p}\cdot  \exp \left(  (-\frac{\tilde{c}_1}{m}  -  \frac{\tilde{c}_2}{(p-1)m} ) ||f'||^2_{L_2 [1,\infty)} \right).
\end{align*}
Optimizing in $p$ shows that the best choice is 
\begin{align*}
 p := 1 + \sqrt{ \frac{\tilde{c}_2 ||f'||^2_{L_2 [1,\infty)}}{ 2m \ln (1/  \PRO_{Y_n}  (Y_n (t) \leq 1, \text{ }\ln T \leq t \leq T))} } >1.
\end{align*}
Using this and choosing $c_1, c_2$ suitably completes the proof of inequality (\ref{mass2}).

It is left in (\ref{bound}) to show that $\tfrac{1}{p-1}\theta (x,s)$ is almost everywhere bounded away from infinity. More precisely, we will prove $ \tfrac{1}{p-1} f'_n(s)\leq c$  a.e., for $s \in [\ln T,T]$, which follows from  
\begin{align}\label{aux}
   \PRO_{Y_n}  (Y_n (t) \leq 1, \text{ } \ln T \leq  t \leq T) \geq T^{-d} \quad \text{for some } d>0, 
\end{align}
for any $n \in \mathbb{N}$. Indeed, if (\ref{aux}) holds then due to the choice of $p$  we obtain
\begin{align*}
 \frac{1}{p-1} &= \sqrt{\frac{2m \ln (1/  \PRO_{Y_n}  (Y_n (t) \leq 1, \text{ }\ln T \leq t \leq T)) }{\tilde{c}_2 ||f'||^2_{L_2 [1,\infty)}}}  \leq  \sqrt{\frac{2m \ln (T^{-d}) }{\tilde{c}_2 ||f'||^2_{L_2 [1,\infty)}}}  \leq \tilde{c}\cdot  \sqrt{\ln T}.
\end{align*}
Combining this with $f_n'(s) (\ln T)^{1/2} \leq  f'(s) (\ln T)^{1/2} \leq c$  a.e., for $ s \in [\ln T, T]$ (see (\ref{prop2})) we get $ \tfrac{1}{p-1} f'_n(s)\leq   \tilde{c} (\ln T)^{1/2} f'(s) \leq c$  a.e. 
The proof of (\ref{aux}) can be found in the next step. 

% Since $H$ is increasing, it is sufficient to have  a lower bound for
% \begin{align}\label{iterim}
%   \PRO (Y_n (t) \leq 1 , \text{  } \ln T \leq t \leq T)
% \end{align}
% in order to obtain a further estimate of the lower bound of $g_n (T)$. This will be done in the next step.

\paragraph*{Step 2: Proof of (\ref{aux})}
For this purpose,  we represent the process as a sum of independent processes $Y_n (\cdot) \overset{d}{=} X (\cdot) + S_n (\cdot) + f_n (\ln T)$, where $X$ is the original L\'evy process with triplet $(\sigma^2 ,b,\nu(dx))$, $S_n$ is an additive process with triplet $(0,0,\frac{ f_n'(s) |x|}{m} \mathbf{1}_{\{ x \in A\}} \nu(dx) ds)$. Again, by homogenization there exists a L\'evy process $\tilde{S}$ with triplet $(0,0, \tfrac{|x|}{m} \mathbf{1}_{\{ x \in A\}} \nu(dx))$ such that $S_n (\cdot) = \tilde{S} (f_n (\cdot) -f_n(\ln T))$ f.d.d.  Note that $\tilde{S}$ is a martingale with some finite exponential moment since $A$ is bounded away from minus infinity. 

Since $f_n (\ln T) \leq \kappa \ln T $ then according to Lemma \ref{const1c} we have, for $T$ sufficiently large,
\begin{align*}
 \PRO (X  (t)  \leq  -f_n (\ln T), \text{ } \ln T \leq t \leq T) \geq  9 T^{-\kappa \ln 3} \cdot  \PRO (X  (t)  \leq 1, \text{ } \ln T \leq t \leq 4 T^{1 + \ln 2}).
\end{align*}
 Combining this with the indendence of $X$ and $\tilde{S}$ yields
\begin{align*}
 \PRO & \Bigl(Y_n (t) \leq 1, \text{ } \ln T \leq  t \leq T \Bigr) \\
&= \PRO \Bigl(X  (t) + \tilde{S} (f_n (t)-f_n(\ln T)) + f_n (\ln T) \leq 1, \text{ } \ln T \leq t \leq T \Bigr) \\
& \geq \PRO \Bigl(X  (t)  \leq  -f_n (\ln T), \text{ } \ln T \leq t \leq T \Bigr) \cdot \PRO \Bigl( \tilde{S} (f_n (t)-f_n(\ln T)) \leq 1, \text{ } \ln T \leq t \leq T \Bigr) \\
& \geq 9 T^{-\kappa \ln 3} \cdot \PRO \Bigl(X  (t)  \leq 1 , \text{ } \ln T \leq t \leq 4 T^{1+\ln 2}\Bigr)  \cdot \PRO \Bigl( \tilde{S} (t) \leq 1, \text{ } 0 \leq t \leq f_n (T) -f_n ( \ln T) \Bigr) \\
& \geq  9 T^{-\kappa \ln 3} \cdot \PRO \Bigl(X  (t)  \leq 1 , \text{ } 1 \leq t \leq T^{2} \Bigr)  \cdot \PRO \Bigl( \tilde{S} (t) \leq 1, \text{ } 0 \leq t \leq \kappa T\Bigr) \\
&\geq T^{-2 \delta -1/2 - \kappa \ln 3 +o(1)},
\end{align*}
where we used in the last step the fact that the survival exponent of a L\'evy martingale with finite variance is equal to $1/2$ (see \cite{feller}, Chapter  XII).

\paragraph*{Step 3: Proof of (\ref{lowerboundmass2}) } 
Having deduced  (\ref{mass2}) we will prove the following lower bound, for any $n \in \mathbb{N}$,
\begin{align}\label{lowerboundmass2}
 \PRO  (Y_n (t) \leq 1, \text{ } \ln T \leq t \leq T) \geq g_{n+1} (T) \cdot \beta,
\end{align}
where $\beta>0$ is a constant specified later.

%\subparagraph{Homogenization:} 
We represent the process $Y_n$ as a sum of independent processes $Y_n (\cdot) \overset{d}{=} X (\cdot) + Z_n(\cdot) + f_n (\ln T)$, where $Z_n$ is an additive process with triplet \\
$(0,0, \tfrac{f'_n (s) |x|}{m} \mathbf{1}_{\{ x \in A\}} \nu(dx) ds)$. 
Due to the L\'evy-Khintchine formula and
\begin{align*}
 f_n (t) -f_n (\ln T) = \int^t_{\ln T} f'_n (s)ds = \int^t_{0} f'_n (s)ds,
\end{align*}
there exists a L\'evy process $\tilde{Z}$ with triplet $(0,0,\tfrac{ |x|}{m} \mathbf{1}_{\{ x \in A\}} \nu(dx))$ such that \\ $Z_n (\cdot) = \tilde{Z} (f_n (\cdot)-f_n (\ln T))$ in f.d.d.  Note that $\tilde{Z}$ is a L\'evy martingale with some finite exponential moment, since $A$ is compact in $(-\infty,0)$ and the characteristic exponent of $\tilde{Z}$ has the following representation 
\begin{align*}
 \Psi (u) %&= \int_{\mathbb{R} } (1-e^{iux} +  \mathbf{1}_{\{ |x| \leq 1\}}  iux)  \tfrac{|x|}{m} \mathbf{1}_{\{ x \in A\}}  \nu (dx)\\
&= \int_{\mathbb{R} } (1-e^{iux} +  iux)  \tfrac{|x|}{m} \mathbf{1}_{\{ x \in A\}}  \nu (dx)
\end{align*}
and  L\'evy measure satisfying $\int ( |x| \wedge x^2 )  \tfrac{|x|}{m} \mathbf{1}_{\{ x \in A\}}  \nu(dx)< \infty$. Thus,
\begin{align*}
  \PRO & (Y_n (t) \leq 1, \text{ } \ln T \leq t \leq T) =\PRO   (X (t) + Z_n (t) \leq 1 - f_n (\ln T), \text{ }\ln T \leq t \leq T) \\
&\quad = \PRO (X (t) + \tilde{Z}(f_n (t) - f_n (\ln  T)) \leq 1 - f_n (\ln T), \text{ }\ln T \leq t \leq T).
\end{align*}

%\subparagraph{Independence:}
Recall that there exists $\kappa >0$ such that $f(T) \leq \kappa T$, for $T$ sufficiently large (see (\ref{prop1})). Using the independence of $X$ and $\tilde{Z}$  we can write, for $T$ sufficiently large,
\begin{align}\label{prod}
% &\pr{}{Y_n (t) \leq 1 , \text{ } \ln T \leq t \leq T}\notag \\
& \pr{}{X (t) + \tilde{Z} (f_n (t)-f_n (\ln T)) \leq 1 - f_n(\ln T), \text{ }  \ln T \leq t \leq T} \notag \\
&\geq \pr{}{ X (t) \leq 1 -  \max\{ 1, (f_n (t)-f_n (\ln T))^{2/3}\} -f_n (\ln T), \text{ }  \ln T \leq t \leq T}\notag \\
&\quad \cdot  \pr{}{ \tilde{Z} (f_n (t)-f_n (\ln T)) \leq \max\{ 1, (f_n (t)-f_n (\ln T))^{2/3}\}  ,\text{ }  \ln T \leq t \leq T} \notag \\
&\geq   \pr{}{ X (t) \leq 1 -  f_{n+1} (t) , \text{ }  \ln T \leq t \leq T} \cdot \pr{}{ \tilde{Z} (t) \leq  \max\{1, t^{2/3} \} , \text{ }  0 \leq t \leq \kappa T}  \notag \\
&=  g_{n+1} (T) \cdot \pr{}{ \tilde{Z} (t) \leq  \max\{1, t^{2/3} \} , \text{ }  0 \leq t \leq \kappa T},
\end{align}
where we used in the second step that $f_n (T)-f_n (\ln T) \leq f (T) \leq \kappa T$, for $T$ sufficiently large (see \ref{findu}).
Since $\tilde{Z}$ is a martingale with some exponential moment and using  Lemma \ref{associated} and \ref{coup} implies, for $0<\beta<1$ suitably chosen and $\beta = \beta (\tilde{Z})$, 
\begin{align}\label{bteil2}
 & \pr{}{ \tilde{Z} (t)  \leq  \max\{1, t^{2/3} \} , \text{ } 0 \leq t \leq \kappa T} \notag \\
&\geq \pr{}{ \tilde{Z} (t)  \leq  1, \text{ } 0 \leq t \leq 1}\pr{}{ \tilde{Z} (t)  \leq  \max\{1, t^{2/3} \} , \text{ } 1 \leq t \leq \kappa T} \gtrsim  \beta,
\end{align} 
where $\pr{}{ \tilde{Z} (t)  \leq  1, \text{ } 0 \leq t \leq 1} >0$ is constant depending on $\tilde{Z}$.  Combining (\ref{bteil2}) with (\ref{prod}) shows (\ref{lowerboundmass2}).

\paragraph*{Step 4: Proof of (\ref{iteration})}
% In order to apply the iteration, our next goal is to show that, for every $n \in \mathbb{N}$,
% \begin{align}\label{goal}
%  g_n(T) \geq H( g_{n+1} (T) \cdot \beta )
% \end{align}
% holds, where 
% \begin{align*}
% g_n (T) :=  \PRO (X(t) \leq 1 - f_n (t), \text{  } \ln T \leq t \leq T).
% \end{align*}
Plugging (\ref{lowerboundmass2}) into (\ref{mass2}) and using that $H$ is monotone on $(0,1]$ we obtain, for any $n \in \mathbb{N}$, that
\begin{align}\label{iteration}
g_{n} (T)& \geq \beta \cdot g_{n+1} (T)  \cdot \exp \left( -\sqrt{c_1 ||f'||^2_{L_2 [1,\infty)} \ln (1 /( g_{n+1} (T) \cdot \beta) )} - c_2 ||f'||^2_{L_2 [1,\infty)} \right)    \notag \\
& = H(g_{n+1} (T) \cdot \beta),
\end{align}
which provides the iteration rule.

\paragraph*{Step 5: Proof  of (\ref{endpoint})}
The aim of this step is to find a number $n(T)$ depending on $T$ such that 
\begin{align}\label{endpoint}
 g_{n(T)} (T) \geq T^{-\delta +o(1)} \cdot G(\ln T).
\end{align}
This inequality presents our end point of the iteration.

% \paragraph{Number of iteration steps}
Our first goal  of this step is to set the number of iteration steps, depending on $T$, such that eventually the boundary is larger than $-1 - f (\ln T) - n(T)$. Recall that  $f(T) \leq \kappa T$. We choose, for $T$ sufficiently large,
\begin{align*}
 n(T) := \left\lceil  \frac{\ln (\ln (\kappa T)/\ln(2))}{\ln(3/2)} \right\rceil,
\end{align*}
and thus, for $T$ sufficiently large,
\begin{align} \label{absch}
g_{n(T)} (T)% &=  \PRO  (X(t) \leq 1 - f_{n(T)} (t) , \text{  } \ln T \leq t \leq T) \notag \\
& \geq  \PRO (X(t) \leq -1 - f (\ln T) - n(T), \text{  } \ln T \leq t \leq T),
\end{align}
% Indeed, 
% \begin{align*}
%     n(T) >  \frac{\ln (\ln (2)/\ln(f(T)))}{\ln(2/3)} 
% \end{align*}
% and thus,
% \begin{align*}
%  f_{n(T)} (t) < 2 + n(T) + f(\ln T), \quad \text{ for all }t \leq T,
%  \end{align*}
since $f$ is non-decreasing and inequality (\ref{findu}) holds.
%  Furthermore, we used that $f(\ln T)>1$, for $T$ sufficiently large. On the other hand, if $f(\ln T) <1$, for all $T$, we have\\ $\sup_{t\geq0} |f(t)| < \infty$ then applying Lemma \ref{const1c} already proves the theorem.

% \paragraph{Asymptotic rate of the end point:} 

Next, we show (\ref{endpoint}) to obtain the asymptotic rate of the end point. Recall that $f'(t) \searrow 0$, for $t\rightarrow \infty$, and $n(T) \leq b_1 (\ln (\ln T))$, for $b_1 >0$ suitably chosen. Define $k(T):=2+f'(1) + b_1 \ln (\ln T)$. 
Since $(X (t))_{t\geq0}$ has stationary and independent increments we have due to (\ref{absch})
\begin{align}
&g_{n(T)} (T) \geq \PRO \Bigl( X(t) \leq -1 - f (\ln T) -n(T), \text{  } \ln T \leq t \leq T  \Bigr) \notag \\
&\geq \PRO \Bigl( \{X(t) \leq -1  - f (\ln T) - n(T), \text{  } \ln T \leq t \leq T\} \cap \{X(\ln T-1) \leq 1 -f (\ln T-1)\} \Bigr)  \notag \\
&\geq \PRO \Bigl( \{X(t) -X(\ln T-1) \leq -2 - f (\ln T) - f (\ln T-1) -n(T) , \text{  } \ln T \leq t \leq T \}  \notag \\
&\quad  \cap \{ X(\ln T-1) \leq 1 -f (\ln T-1)\} \Bigr)  \notag \\
&\geq \PRO \Bigl( \{X(t) -X(\ln T-1) \leq -k(T) ,  \ln T \leq t \leq T \} \cap \{ X(\ln T-1) \leq 1 -f (\ln T-1)\} \Bigr)  \notag \\
&\geq \PRO \Bigl( X(t) \leq -k(T)  , \text{  } 1\leq t \leq T - \ln T +1  \Bigr)  \notag \\
&\quad \cdot \PRO \Bigl( X(t) \leq 1 -f(t), \text{  } \ln (\ln T ) \leq t \leq \ln T -1 \Bigr) \notag \\
&\geq  3^{-k(T)-2} \cdot \PRO \Bigl( X(t) \leq 1  , \text{  } 1 \leq t \leq (T - \ln T +1) \cdot 2^{k(T)+2}  \Bigr) \notag \\
&\quad \cdot \PRO \Bigl( X(t) \leq 1 -f(t), \text{  } \ln (\ln T ) \leq t \leq \ln T \Bigr) \notag \\
&=   T^{-\delta + o(1)} \cdot G (\ln T), \notag
\end{align}
where the second last step follows analogously to  Lemma \ref{const1c} in spite of the negative boundary since $\nu (\mathbb{R}_-)>0$ and the considered time interval of the one-sided exit problem does not contain zero. In the last step we used assumption (\ref{assum}). Hence, we have (\ref{endpoint}).

\paragraph*{Step 6: Proof of (\ref{subB})}
% Application of the iteration
In this step we combine inequality (\ref{iteration}) with (\ref{endpoint}) to obtain finally inequality (\ref{subB}). 

Since $H'>0$ on $(0,1]$, inequality (\ref{iteration}) implies  $g_0 (T) \geq H^{n(T)}_\beta (g_{n(T)} (T) )$. Our first goal is to calculate  $H^{n(T)}_\beta (g_{n(T)} (T) ) $ with the help of (\ref{endpoint}).
We start showing by induction that
\begin{align}\label{indi}
 &H^n_\beta (x) \geq  W_{n} (x) \cdot  \exp \left( -  n \sqrt{c_1    ||f'||^2_{L_2 [1,\infty)} \ln \left(  W_{n} (x)^{-1} \cdot Z_{n} (x) \right)}  \right),
\end{align}
for all $n \geq 1$ and $x \in (0,1]$, where 
\begin{align*}
 W_{n} (x) : = x   \cdot \beta^n  \cdot \exp \left(  - n \cdot c_2 ||f'||^2_{L_2[1,\infty)}   \right) ,
\end{align*}
and
\begin{align*}
 Z_{n} (x):=   \exp \left( (n-1)  \sqrt{c_1    ||f'||^2_{L_2 [1,\infty)} 2^{n-2} \ln \left( x^{-1} \beta^{-2}  \right) } - c_2 ||f'||^2_{L_2[1,\infty)}   \right).
\end{align*}
Indeed, we have, for $n=1$, that
\begin{align*}
 H^1_\beta (x) &= H (x \cdot \beta )  = W_{1} (x)  \cdot \exp \left( -   \sqrt{c_1    ||f'||^2_{L_2 [1,\infty)} \ln \left( \left( W_{1} (x)\right)^{-1} Z_{1}(x)  \right)} \right).
\end{align*}
Assume now that (\ref{indi}) holds, for $n-1$. Note that, for $x$ sufficiently small, we have
\begin{align*}
H(x) \geq x^2.
\end{align*}
First, we get
\begin{align*}
 W_{n-1} &\Bigl( H(x \cdot \beta ) \Bigr)  = W_{n} (x) \cdot \exp \left(  -\sqrt{c_1 ||f'||^2_{L_2 [1,\infty)} \ln \left(  x^{-1} \beta^{-1}   \right)  }   \right).
\end{align*}
Hence, we obtain,  for $x \in (0,1]$, that
\begin{align*}
& W_{n-1}  \bigl( H(x \cdot \beta ) \bigr)^{-1} \cdot Z_{n-1} \left(H(x \cdot \beta) \right) \\
& \leq  \frac{1}{ W_{n} (x)}  \cdot \exp \left(  \sqrt{c_1 ||f'||^2_{L_2 [1,\infty)} \ln \left(  x^{-1} \beta^{-1}  \right)  } + (n-2) \sqrt{c_1 ||f'||^2_{L_2 [1,\infty)} 2^{n-3} \ln \left(  x^{-2} \beta^{-4}  \right)  }   \right) \\
&\leq \bigl( W_{n} (x)\bigr) ^{-1}  \cdot Z_{n} (x),
\end{align*}
since $\beta \leq 1$. This implies, for $x$ sufficiently small,
\begin{align*}
 H^n_\beta (x) &= H^{n-1}_\beta \Bigl( H(x \cdot \beta) \Bigr) \\
&\geq W_{n-1} \Bigl( H(x \cdot \beta) \Bigr)\\
&\quad \cdot  \exp \left(- (n-1) \sqrt{  c_1    ||f'||^2_{L_2 [1,\infty)} \ln \left(   W_{n-1}  \Bigl( H(x \cdot \beta) \Bigr)^{-1}  Z_{n-1} \Bigl( H(x \cdot \beta ) \Bigr) \right) }    \right) \\
 & \geq W_{n} (x) \cdot \exp \left(  -\sqrt{c_1 ||f'||^2_{L_2 [1,\infty)} \ln \left(  x^{-1} \beta^{-1}  \right)  }   \right)\\
& \quad \cdot \exp \left(- (n-1) \sqrt{  c_1    ||f'||^2_{L_2 [1,\infty)} \ln \left(  \Bigl(  W_{n} (x) \Bigr) ^{-1}  Z_{n} (x)\right) }    \right) \\
& \geq  W_{n} (x) \cdot \exp \left(- n \sqrt{  c_1    ||f'||^2_{L_2 [1,\infty)} \ln \left(  \Bigl( W_{n} (x) \Bigr)^{-1}  Z_{n} (x)   \right)}    \right),
\end{align*}
where we used in the last step that, for $n\geq 2$,
\begin{align*}
&\Bigl( W_{n} (x) \Bigr)^{-1}  Z_{n} (x) \\
&= x^{-1} \beta^{-n} \cdot \exp \left(  (n-1) c_2 ||f'||^2_{L_2 [1,\infty)} \right) \cdot  \exp \left( (n-1)  \sqrt{c_1    ||f'||^2_{L_2 [1,\infty)} 2^{n-2} \ln \left( x^{-1} \beta^{-2}  \right) } \right) \\
& \geq x^{-1} \beta^{-1}.
\end{align*}
Recall that $n(T) \leq b_1 \bigl( \ln (\ln T) \bigr)$ and $g_{n(T)} (T) \leq T^{-\delta +o(1)}$, for $b_1 = 5/2$. Then, we obtain that
\begin{align}\label{Z1}
 Z_{n(T) } (g_{n(T)} (T) )  &\leq  \exp \left(  b_1 \Bigl( \ln (\ln T) \Bigr)  \sqrt{c_1    ||f'||^2_{L_2 [1,\infty)} \delta \cdot 2^{-2} (\ln T)^{b_1 \ln 2 } \ln \left( T \beta^{-2}  \right) } \right) \notag \\
&\leq \exp \left(  b_1( \ln \ln T) \cdot (\ln T)^{7/5}   \right)  \leq   \exp \left(   (\ln T)^{3/2}  \right)   \notag \\
&= T^{\sqrt{\ln T}},
\end{align}
for $T$ sufficiently large, and
\begin{align}\label{Z2}
 \exp \left(  -  n(T) \cdot c_2 ||f'||^2_{L_2[1,\infty)}   \right)  = T^{o(1)}  \geq   \exp \Bigl( - \tfrac{\delta}{2} \cdot (\ln T)    \Bigr)  \geq g_{n(T)} (T) .
\end{align}
Putting (\ref{Z1}) and (\ref{Z2}) into (\ref{indi}) we obtain, for $b_2 >0 $ suitably chosen, that
\begin{align*}
  H^{n(T)}_\beta (g_{n(T)} (T) ) &\geq g_{n(T)} (T) \cdot \beta^{n(T)}   \cdot \exp \left(  - n(T) \cdot c_2 ||f'||^2_{L_2[1,\infty)}   \right)  \\
& \quad  \cdot  \exp \left( -  n(T) \sqrt{c_1    ||f'||^2_{L_2 [1,\infty)} \ln \left( g_{n(T)} (T)^{-2} \beta^{-n(T)}T^{\sqrt{\ln T}}  \right)}  \right)  \\
&\geq g_{n(T)} (T) \cdot \beta^{n(T)}   \cdot \exp \left(  - n(T) \cdot c_2 ||f'||^2_{L_2[1,\infty)}   \right)  \\
& \quad  \cdot  \exp \left( -   3 \cdot n(T)\cdot (\ln T)^{3/4}  \sqrt{c_1    ||f'||^2_{L_2 [1,\infty)} }  \right)  \\
%&\geq g_{n(T)} (T) \cdot   \exp \left(- b_2 \left( \bigl( \ln (\ln T) \bigr) +  \bigl( \ln (\ln T) \bigr)  \cdot  (\ln T)^{1/2}  \right)   \right)  \\
&\geq g_{n(T)} (T)   \cdot  \exp \left(- b_2  (\ln T)^{4/5}     \right)  \\
&\geq g_{n(T)} (T)  \cdot  T^{ o(1)} .
\end{align*}
Combining this with (\ref{endpoint}) and an $n(T)$-times iteration of (\ref{iteration}) yields
\begin{align*}
  g_0 (T) &= \PRO (X(t) \leq 1 - f(t), \text{  } \ln T \leq t \leq T)  \geq H_\beta^{n(T)} (g_{n(T)} (T)) = T^{-\delta + o(1)} \cdot G(\ln T), 
\end{align*}
which completes the proof of (\ref{subB}) provided (\ref{aux}) holds.

%%%%%%%%%%%%%%%%%%%%%%%%%%%%%%%%%%%%%%%%%%%%%%%%%%%%%%%%%%%%%%%%%%%%%%%%%%%%%%%%%%%%%%%%%%%%%%%%%%%%%%%%%%%%%%%%%%%%%%%%%%%%%%%%%%%%%%%%%%%%%%%%%%%%%%%%%%%%%%%%%%%%%%%%%%%%%%%%%%%%%%%%%%%%%%%%%%%%%%%%%%%%%%%%%%%%%%%%%%%%%%%%%%%%%%%%%%%%%%%%%%%%%%%%%%%%%%%%%%%%%%%%%%%%%%%%%%%%%%%%%%%%%%%%%%%%%%%%%%%%%%%%%%%%%%%%%%%%%%%%%%%%%%%%%%%%%%%%%%%%%%%%%%%%%%%%%%%

%%%%%%%%%%%%%%%%%%%%%%%%%%%%%%%%%%%%%%%%%%%%%%%%%%%%%%%%%%%%%%%%%%%%%%%%%%%%%%%%%%%%%%%%%%%%%%%%%%%%%%%%%%%%%%%%%%%%%%%%%%%%%%%%%%%%%%%%%%%%%%%%%%%%%%%%%%%%%%%%%%%%%%%%%%%%%%%%%%%
\section{Proof of Theorem \ref{mblp} (positive boundaries)}\label{proof1}

Since $f$ is positive, our quantity is trivially bounded from below as follows
\begin{align*}
  \PRO (X(t) \leq 1+ f(t), \text{ } 0\leq t \leq T) \geq \PRO (X(t) \leq 1, \text{ } 0\leq t \leq T) = T^{-\delta + o(1)}.
\end{align*}

Our goal is to show 
\begin{align}\label{goalpos}
  \PRO (X(t) \leq 1+ f(t), \text{ } 0\leq t \leq T) \leq T^{-\delta + o(1)}.
\end{align}
% The proof of the upper bound is divided into a sequence of subsections.

\subsection{Preliminaries}\label{Pre}

In the following proof we can assume that $T>1$. We can write
\begin{align*}
 \PRO (X(t) \leq 1+ f(t), \text{ } 0\leq t \leq T) \leq \PRO (X(t) \leq 1+ f(t), \text{ } 1 \leq t \leq T). 
\end{align*}
Hence, as from now we consider the time interval $[1,T]$.\\
\textbf{Auxiliary function $H$ for the iteration:} We define 
\begin{align*}
 H(x) := x \exp \left(\sqrt{c_1 ||f'||^2_{L_2 [1,\infty)}  \ln(1/x) }   \right), \quad  x \in (0,1].
\end{align*}
Note that  $H'(x)>0$ on $(0,1)$.
Furthermore, define $H^{0}_2 (x) :=H(2x)$ and, for $i\geq 1$,
\begin{align*}
 H^{i}_2 (x) := H \bigl(2 H^{i-1}_2 (x) \bigr).
\end{align*}
$H^i_2$ is well defined since $H(x) \in (0,1]$ for  $x \in (0,1]$.\\
\textbf{Auxiliary function $f_n$ for the iteration:} Define
$f_0 (t) := \max\{f(\ln T),f(t)\}$ and, for $n\geq 1$, $f_n (t) := f (\ln T) + n \kappa_{\delta} \ln T + n (\ln T)^5$, for $t\leq \ln T$, and, for $t > \ln T$, 
\begin{align*}
   f_n (t) :=  f_{n-1} (\ln T) + n \kappa_{\delta} \ln T + \max \left\{ (\ln T)^5, \left( f_{n-1} (t) - f_{n-1} (\ln T) \right)^{3/4} \right\},
\end{align*}
where $\kappa_{\delta} >0$ is constant specified later.
By induction it follows, for $t > \ln T$ and $n\geq 0$, that
\begin{align}\label{abs}
 f_n (t) \leq f(\ln T) + n \kappa_{\delta} \ln T + (n-1) (\ln T)^5 + \max \left\{ (\ln T)^5, f(t)^{(3/4)^n} \right\}.
\end{align}
Furthermore, define $\tilde{t}_{T,n} := \inf \{ t \geq 0: (\ln T)^5 <  \left( f_{n-1} (t) - f_{n-1} (\ln T) \right)^{3/4} \}$. Note that, for $n\geq 1$,
\begin{align*}
f_n'(t) = 
 \begin{cases}
0 ,& t<  \tilde{t}_{T,n}, \\
 \frac{3}{4} \left( f_{n-1} (t) - f_{n-1} (\ln T) \right)^{-1/4}  f_{n-1}' (t), &     t > \tilde{t}_{T,n}  . 
\end{cases}
\end{align*}
Since $\left( f_{n-1} (t) - f_{n-1} (\ln T) \right)^{3/4} > (\ln T)^5$ we get again by induction
\begin{align}\label{incr}
  f_n' (t) \leq f' (t) \quad \text{ a.e.}
\end{align}
Note that $\tilde{t}_{T,n}$ is non-decreasing in $n$. Without loss of generality we can assume that $\tilde{t}_{T,n}\geq 1$, for all $n>0$ and $T$ sufficiently large. Otherwise, we choose $T$ sufficiently large such that $ \left( f_{n-1} (1) - f_{n-1} (\ln T) \right)^{3/4} < (\ln T)^5$ and thus,  $\tilde{t}_{T,n}\geq 1$.

\subsection{Iteration; Proof of (\ref{goalpos})}

% \subsubsection{Goal setting}
% Our first goal is to show that, for every $n \in \mathbb{N}$,
% \begin{align}\label{goalset}
%  g_n(T) \leq H  \bigl( 2 g_{n+1} (T) \bigr)
% \end{align}
% holds, where 
First, define
\begin{align*}
g_n (T) :=  \PRO (X(t) \leq 1 + f_n (t), \text{  } 1 \leq t \leq T). 
\end{align*}

%The following steps will be conducted for any $n \in \math{N}$.

\paragraph*{Step 1: Proof of (\ref{mass})}
By using a change of measure the aim of this step is to show the following inequality:
\begin{align}\label{mass}
g_n(T)% &= \PRO (\tilde{X}_n (t) \leq 1 , \text{  } 1 \leq t \leq T) \notag \\
&\leq \PRO (Y_n (t) \leq 1, \text{  } 1 \leq t \leq T) \notag \\
& \quad \cdot \exp \left( \sqrt{c_1 ||f '||^2_{L_2 [1,\infty)} \ln (1/ \PRO (Y_n(t) \leq 1 , \text{  } 1 \leq t \leq T) )} \right)\notag \\
&= H \Bigl( \PRO (Y_n (t) \leq 1 , \text{  } 1 \leq t \leq T)  \Bigr),
\end{align}
where $c_1>0$ is a constant depending on $\nu$ and $f$ that is chosen later on.

In the same way as previously, we can assume that $\nu ((0,1]) >0$.

Since $\nu((0,1])>0$, we can choose a compact set $A \subseteq (0,1]$  such that 
\begin{align*}
 0< \int_A x^2 \nu(dx)=: m < \infty.
\end{align*}

 Let $\tilde{X}_n$ and $Y_n$ be two additive processes with triplets $(\sigma^2 , f_{\tilde{X}_n} (t), \nu(dx)ds )$ and  \\
$(\sigma^2 , f_{Y_n} (t) , (1+ \frac{f_n'(s) x}{m} \mathbf{1}_{\{x \in A \}}) \nu(dx)ds )$ respectively, where $f_{Y_n} (t) := b\cdot t-f_n(1)$ and $f_{\tilde{X}_n} (t) := b\cdot t -f_n (t)$.

The same arguments as previously implies that $\PRO_{\tilde{X}_n} |_{\FF_T}$ and $\PRO_{Y_n} |_{\FF_T}$ are absolutely continuous with  $\frac{d \Lambda_{Y_n} (x,s)}{\nu(dx)ds }  = e^{\theta (x,s)}$, where $\theta(x,s) := \ln (1 + \tfrac{ f_n'(s)x}{m} \mathbf{1}_{\{x \in A\}})$, for all $ s \in [0,T]$ and $x \in \mathbb{R}$, and $\Lambda_{Y_n} (dx,ds) := \exp(\theta(x,s) )\nu(dx)ds$

Next, we prove inequality (\ref{mass}).

Note that $\theta (x,s) =0$, for $s \in[0,1]$ and $x \in \mathbb{R}$. Because of Theorem \ref{gir} and the density transformation formula (\ref{dens}) we have
\begin{align}\label{dentran}
 \PRO & (\tilde{X}_n (t) \leq 1, \text{ } 1 \leq t \leq T) = \ER_{\tilde{X}_n} ( \mathbf{1}_{\{\tilde{X}_n(t) \leq 1, \text{ } 1 \leq  t \leq T\}}) \\
%&=  \ER_{Y_n} \left( \mathbf{1}_{\{Y_n (t) \leq 1, \text{ } 1 \leq  t \leq T\}}  \frac{d\PRO_{\tilde{X}_n}}{d\PRO _{Y_n}} (Y_n (\cdot)) \right) \\
 %&= \ER_{Y_n} \left( \mathbf{1}_{\{Y_n (t) \leq 1, \text{ } 1 \leq  t \leq T\}} e^{ - \int_{1}^T \int_{\mathbb{R} }  \theta (x,s)  \bar{N}_{Y_n} (dx,ds)} \right) \cdot e^{ \int_{1}^T \int_{\mathbb{R} }  \left( e^{\theta(x,s)} -1- \theta(x,s) e^{\theta(x,s)} \right) \nu(dx) ds } \\
&= \ER_{Y_n} \left( \mathbf{1}_{\{Y_n (t) \leq 1, \text{ } 1 \leq  t \leq T\}} e^{ - \int_{1}^T \int_{\mathbb{R}}  \theta (x,s)  \bar{N}_{Y_n} (dx,ds)} \right) \cdot e^{- \int_{1}^T \int_{\mathbb{R} }  g \left(\tfrac{f_n'(s) x}{m} \mathbf{1}_{x\in A } \right) \nu (dx)ds },
\end{align}
where $g(u) := (1+u ) \ln (1+u) - u$, $u\geq 0$. Since $g(u)\geq0$, for $u\geq0$, we obtain that
\begin{align*}
 e^{- \int_{1}^T \int_{\mathbb{R}}  g (\tfrac{f_n'(s) x}{m} \mathbf{1}_{x\in A } ) \nu (dx)ds } \leq 1.
\end{align*}

Let $p>1$ and $1/p +1/q =1$. Applying H\"older's inequality in (\ref{dentran}) yields that
\begin{align}\label{one}
&\PRO  (\tilde{X}_n (t) \leq 1, \text{ } 1 \leq  t \leq T) \notag \\
%& \leq  \ER_{Y_n} \left(\mathbf{1}_{\{Y_n (t) \leq 1, \text{ } 1 \leq  t \leq T\}} e^{-\int_{1}^T \int_{\mathbb{R}}  \theta (x,s) \bar{N}_{Y_n} (dx,ds)  }  \right) \notag \\
&\leq \left(\PRO(Y_n (t) \leq 1, \text{ } 1 \leq  t \leq T) \right)^{1/p}   \cdot \left(  \ER_{Y_n} \left(\exp \left(- q \int_{1}^T \int_{\mathbb{R}}  \theta (x,s) \bar{N}_{Y_n} (dx,ds)  \right) \right) \right)^{1/q} .
\end{align}
Let us consider the second term in (\ref{one}). Due to the density transform formula (\ref{dens2}) we have
 \begin{align*}
 &  \ER_{Y_n} \left(e^{-q \int_1^T \int_{\mathbb{R} }  \theta (x,s) \bar{N}_{Y_n} (dx,ds)  }  \right) \\
% & =  \ER_{\tilde{X}_n} \left(e^{ \int_1^T \int_{\mathbb{R} } -q \theta (x,s) ( N (dx,ds) - \Lambda_{Y_n} (dx,ds)  ) +  \theta (x,s) ( N (dx,ds) - \nu(dx) ds) }  \right)\\
%  & \quad \cdot    e^{ - \int_1^T \int_{\mathbb{R} }  (e^{\theta (x,s)} - 1- \theta(x,s))  \nu(dx) ds }  \\
&=   \ER_{\tilde{X}_n} \left(e^{ \int_1^T \int_{\mathbb{R} } (-q+1)  \theta (x,s) ( N (dx,ds) - \nu(dx) ds) }  \right)\\
 & \quad \cdot   \exp{ \left( \int_1^T \int_{\mathbb{R} }  (-q\theta(x,s)+ q\theta (x,s) e^{\theta(x,s)}-e^{\theta (x,s)} + 1+ \theta(x,s))  \nu(dx) ds \right) }   \\
&=   \exp{ \left(  \int_1^T \int_{\mathbb{R} }  (e^{(-q +1)\theta (x,s)} -1 - (-q +1) \theta(x,s))  \nu(dx) ds \right)}   \\
 & \quad \cdot   \exp{ \left(  \int_1^T \int_{\mathbb{R}}  (-q\theta(x,s)+ q\theta(x,s) e^{\theta(x,s)}-e^{\theta (x,s)} + 1+ \theta(x,s))  \nu(dx) ds \right) }   \\
&=\exp{\left(  \int_1^T \int_{\mathbb{R}} e^{\theta(x,s)} (e^{-q\theta (x,s)} -1 + q \theta(x,s))  \nu(dx) ds \right) } ,
\end{align*} 
where we used as in the proof of Theorem \ref{falling} a modification of Lemma 33.6 of \cite{sato} in the second step. Again, the difference between \cite{sato} and our case consists in the consideration of time-inhomogeneous processes in contrast to time-homogeneous processes used in \cite{sato}. 

Taylor's expansion implies $ e^{-q \theta (x,s)}+ q \theta (x,s) - 1 \leq \tfrac{1}{2} q^2 \theta (x,s)^2$,  for all $x \in \mathbb{R}$ and $s \in [1,T]$.  Since $\theta$ is bounded away from infinity we have $\exp(\theta(x,s))<\tilde{c}_1$, for some $\tilde{c}_1>0$, and thus,
\begin{align}
\frac{1}{q} & \int_{1}^T \int_{\mathbb{R}} e^{\theta(x,s)} ( e^{-q \theta (x,s)}+q \theta (x,s) - 1 )  \nu (dx)ds \leq q \int_{1}^T \int_{\mathbb{R}} \tfrac{\tilde{c}_1}{2} \theta (x,s)^2  \nu (dx)ds \notag \\
& \leq \frac{q \cdot \tilde{c}_1}{2 m^2} \int_{1}^T f_n' (s)^2 ds \cdot \int_A x^2 \nu(dx) \leq \frac{q \cdot \tilde{c}_1}{2 m}  ||f'||^2_{L_2 [1,\infty)},\notag
\end{align}
having also used (\ref{incr}). 
Plugging this into (\ref{one}) yields
\begin{align*}
 g_n(T) &= \PRO  (\tilde{X}_n (t) \leq 1, \text{ } 1 \leq t \leq T) \\
& \leq \PRO  (Y_n (t) \leq 1, \text{ } 1 \leq  t \leq T)^{1/p} \cdot \exp \left( \frac{ q \cdot \tilde{c}_1}{2 m} ||f'||^2_{L_2 [1,\infty)} \right).
\end{align*}
Optimizing in $p$ shows that the best choice is 
\begin{align*}
 1/p := 1 - \sqrt{ \frac{\tilde{c}_1 \cdot ||f'||^2_{L_2 [1,\infty)}}{ 2 m \ln (1/  \PRO_{Y_n}  (Y_n (t) \leq 1, \text{ } 1 \leq  t \leq T))} } <1,
\end{align*}
which shows inequality (\ref{mass}) with  $c_1>0$ suitably chosen.

\paragraph*{Step 2: Proof of (\ref{eqn: uperboundmass})}

Having deduced (\ref{mass}) we carry on with the examination of the one-sided exit problem for the process $Y_n$. More  precisely, we will prove the following upper bound, for any $n \in \mathbb{N}$,
\begin{align}\label{eqn: uperboundmass}
 \PRO  (Y_n (t) \leq 1   ,  \text{  } 1 \leq t \leq T) &\leq 2 \cdot g_{n+1} (T).
\end{align}

%paragraph{Homogenization} 
First, we represent the process $Y_n$ as a sum of independent processes $Y_n (\cdot) \overset{d}{=} X (\cdot) + Z_n (\cdot) - f_n (1)$, where  $Z_n$ is an additive process with triplet $(0,0,\frac{ f_n'(s) x}{m} \mathbf{1}_{\{ x \in A\}} \nu(dx) ds)$. 
Due to the L\'evy-Khintchine formula and
\begin{align*}
 f_n(t) - f_n (1) = \int_1^t f_n' (s) ds= \int_0^t f_n' (s) ds
\end{align*}
 there exists a L\'evy process $\tilde{Z}$ with triplet $(0,0, \tfrac{x}{m} \mathbf{1}_{\{ x \in A\}} \nu(dx))$ such that \\
$Z_n (\cdot) = \tilde{Z} (f_n(\cdot)-f_n (1))$ in f.d.d.  Note that $\tilde{Z}$ is a L\'evy martingale with some finite exponential moment,  since $A$ is compact in $(0,\infty)$, and the characteristic exponent of $\tilde{Z}$ has the following representation 
\begin{align*}
 \Psi (u)% &= \int_{\mathbb{R} } (1-e^{iux} +  \mathbf{1}_{\{ |x| \leq 1\}}  iux)  \tfrac{x}{m} \mathbf{1}_{\{ x \in A\}}  \nu (dx)
= \int_{\mathbb{R}} (1-e^{iux} +  iux)  \tfrac{x}{m} \mathbf{1}_{\{ x \in A\}}  \nu (dx)
\end{align*}
and the L\'evy measure satisfies $\int ( |x| \wedge x^2 )  \tfrac{x}{m} \mathbf{1}_{\{ x \in A\}}  \nu(dx)< \infty$. Thus,
\begin{align*}
 \PRO &(Y_n (t) \leq 1, \text{ } 1 \leq t \leq T) =  \PRO \Bigl(X(t) + \tilde{Z}(f_n(t)-f_n(1)) \leq 1 + f_n (1), \text{ } 1 \leq t \leq T \Bigr). 
\end{align*}

%paragraph{Coupling}
 Denote $c_2 := \ER \left( \tilde{Z} (1)^2 \right) <\infty$. Let $B$ be a Brownian motion. Using Lemma \ref{coup2} we can write with a suitable constant $\kappa_{\delta}>0$
\begin{align}\label{auft}
% \PRO &\Bigl(Y_n (t) \leq 1,  \text{ } 1 \leq t \leq T \Bigr)\notag \\
\PRO & \Bigl(X(t) + \tilde{Z} (f_n (t)-f_n (1)) \leq 1 +f_n (1), \text{ } 1 \leq t \leq T \Bigr) \notag \\
%&\leq \PRO \Bigl(X(t) + c_2 B (f_n (t)-f_n (1)) \leq 1 +f_n(1) + \kappa_{\delta} \ln T, \text{ } 1 \leq t \leq T \Bigr) + T^{-1-\delta} \notag \\
&\leq  \PRO \Bigl(X(t) \leq 1 + f_n (1) +  \kappa_{\delta} \ln T - c_2 B (f_n (t)-f_n (1)) , \text{ } 1 \leq t \leq T \Bigr) + T^{-1-\delta}  .
\end{align}

%\paragraph{Properties of Brownian motion}
 In order to apply results of one-sided boundary problems for Brownian motion define the sets
\begin{align*}
 E_n &:= \left\{  c_2 B(f_n (t) -f_n (1)) \geq - \max \{ (\ln T)^5,  (f_n (t)-f_n (1))^{3/4} \}, \text{ } 1 \leq t \leq T   \right\} \\
&\supseteq  \left\{  c_2 B(t) \geq - \max \{ (\ln T)^5,  t^{3/4} \}, \text{ } 0 \leq t \leq  \kappa T   \right\} =: \tilde{E}_n,
\end{align*}
 since $f(T) \leq \kappa T$, for $\kappa >0$ suitably chosen (see (\ref{prop1})). Then due to Lemma \ref{BBgr} and $f_n (1) = f_n (\ln T)$ we obtain that
\begin{align}\label{ab}
  \PRO& \Bigl(X(t) \leq 1 + f_n (1)+ \kappa_{\delta}  \ln T - c_2 B (f_n (t)-f_n (1)) , \text{ } 1 \leq t \leq T \Bigr) \notag \\
&\leq  \PRO \Bigl(X(t) \leq 1 +f_n (1) + \kappa_{\delta}  \ln T  - c_2 B (f_n (t)-f_n (1)) , \text{ } 1 \leq t \leq T; E_n \Bigr) + \pr{}{\tilde{E}_n^c}\notag \\
&\leq  \PRO \Bigl(X(t) \leq 1+f_n (1) + \kappa_{\delta} \ln T +   \max \{ (\ln T)^5,  (f_n (t)-f_n (1))^{3/4} \}, \text{ } 1 \leq t \leq T\Bigr) \notag \\
&\quad  + \exp{ \left(   - \kappa (\ln T)^2 /4\right)}\notag \\
%&=  \PRO \Bigl(X(t) \leq 1 +  f_{n+1} (t) , \text{ } 1 \leq t \leq T \Bigr) + \exp{ \left(   - \kappa (\ln T)^2 /4\right)} \notag \\
&=  g_{n+1} (T) + \exp{ \left(   - \kappa (\ln T)^2 /4\right)} .
\end{align}

\paragraph*{Step 3: Proof of  (\ref{goalset})}
Our next goal is to show the iteration rule that means, for every $n \in \mathbb{N}$,
 \begin{align}\label{goalset}
  g_n(T) \leq H  \bigl( 2 g_{n+1} (T) \bigr)
 \end{align}
Putting (\ref{ab}) and (\ref{auft}) into (\ref{mass}) and using that $H'>0$ on $(0,1]$ we get
\begin{align}
 g_n (T) & \leq  \left[  g_{n+1} (T) +  T^{-1-\delta} + \exp{ \left(   - \kappa  (\ln T)^2 /4\right)} \right]  \notag \\
 & \quad \cdot \exp \left( \sqrt{c_1 ||f'||^2_{L_2 [1,\infty)} \ln (1/  \left[ g_{n+1} (T) +  T^{-1-\delta} + \exp{ \left(   - \kappa  (\ln T)^2 /4\right)} \right]  )}  \right)\notag \\
% &  = \left[  g_{n+1} (T) +  O( T^{-1-\delta})  \right]  \notag \\
%  & \quad \cdot \exp \left( \sqrt{c_1 ||f'||^2_{L_2 [1,\infty)} \ln (1/  \left[ g_{n+1} (T) +  O( T^{-1-\delta })  \right]  )}  \right) \notag \\
& \leq  H(2 g_{ n+1} (T)) , \notag
\end{align}
 where  we used in the last step that $g_{n+1} (T) \geq T^{-1-\delta } +\exp{ \left(   - \kappa  (\ln T)^2 /4\right)}$,  for sufficiently large $T>1$, since
\begin{align}\label{gnT}
  g_{n+1} (T) \geq \PRO (X(t) \leq 1, \text{ } 0\leq t \leq T) = T^{-\delta +o(1)} \geq  T^{-1-\delta } + \exp{ \left(   - \kappa  (\ln T)^2 /4\right)}.
\end{align}
Hence, we have proved (\ref{goalset}).

\paragraph*{Step 4: Proof of (\ref{endpoint1})}
The aim of this step is to find a number $n(T)$ depending on $T$ such that
\begin{align}\label{endpoint1}
 g_{n(T)} (T) \leq T^{-\delta +o(1)}.
\end{align}
which provides the end point of the iteration.
% \subsubsection{Number of iteration steps}
For this purpose, our first goal is to set the number of iteration steps, depending on $T>1$, such that eventually the boundary is smaller than  $1+ (\ln T)^6$. Due to inequality (\ref{prop1}) there exists $\kappa>0$ such that $f(T) \leq \kappa T$. For this purpose, we choose, for $T$ sufficiently large, 
\begin{align*}
 n(T) :=\left\lceil \frac{\ln (\ln (\kappa T)/\ln(2))}{\ln(4/3)}\right\rceil
\end{align*}
 and thus, for $T$ sufficiently large, 
\begin{align*}
 g_{n(T)} (T) %&= \PRO ( X(t) \leq 1 + f_{n(T)} (t) , \text{ } 1 \leq t \leq T) \\
&\leq \PRO \left( X(t) \leq 1 + f (\ln T) + n(T) \cdot (\ln T)^5   , \text{ } 1 \leq t \leq T\right) \\
&\leq  \PRO \left( X(t) \leq 1 + (\ln T)^6, \text{  } 1 \leq t \leq T \right),
\end{align*}
where we used inequality (\ref{abs})  combined with $ f(t)^{3/4^{n(T)}}<2$, for $0 \leq t \leq T$, and that  $f(T)>1$ if $f$ is not bounded away from infinity. On the other hand, if  $\sup_{t\geq 0} |f(t)|  <\infty$, then applying Lemma \ref{const1c} already proves the theorem.

% \subsubsection{Asymptotic rate of the end point:}
Applying Lemma \ref{helpln} implies
\begin{align*}
g_{n(T)} (T) &\leq \PRO  (X (t) \leq 1 + (\ln T)^6, \text{  } 1 \leq t \leq T)  \leq  \frac{ 2 \cdot \PRO \left( X(t) \leq 1, \text{  } 0 \leq t \leq T  \right)}{ \PRO \left( X(t) \leq 1 - t^{1/3}, \text{  } 0 \leq t \leq  ( \ln T)^{21}  \right)}\\
 & =\PRO \left( X(t) \leq 1, \text{  } 0 \leq t \leq T  \right) (\ln T)^{21 \delta + o(1)}, 
\end{align*} 
where we used Theorem \ref{falling} in the last step  and with it the assumption  $\nu (\mathbb{R}_-) >0$. Using now the main assumption (\ref{assum})  yields (\ref{endpoint1}).

\paragraph*{Step 5: Proof of (\ref{goalpos})}
In this step we combine (\ref{goalset}) with (\ref{endpoint1}) to obtain finally inequality (\ref{goalpos}).
For this purpose, we calculate $H^{n(T)}_{2} (2 g_{n(T)} (T))$. First, we show by induction for $x$ sufficiently small that, for any $n \geq 1$,
\begin{align}\label{induct}
 H^n_{2}& (2x) \leq   2^{n} \cdot x \cdot  \exp \left( n  \sqrt{c_1 ||f'||_{L_2 [1,\infty)} \ln(1/x )}  \right) .
\end{align}
Clearly, we get, for $n=1$,
\begin{align*}
H^1_2 (2x) = H(2x) \leq 2 \cdot x \cdot \exp \left(  \sqrt{c_1 ||f'||_{L_2 [1,\infty)} \ln(1/x ) } \right),
\end{align*}
 since $\ln(1/(2 x ))  \leq \ln(1/x ) $. Now, we assume that (\ref{induct}) holds, for $n-1$. Since $H$ is non-decreasing in a  neighbourhood of zero, we have  
 \begin{align*}
 H^n_{2} (2x) &= H (2 H^{n-1} (2x)) \leq H \left( 2^n x \exp \left( (n-1) \sqrt{c_1 ||f'||_{L_2 [1,\infty)} \ln(1/x ) } \right) \right) \\
% & \leq 2^n \cdot x \cdot \exp \left( (n-1) \sqrt{c_1 ||f'||_{L_2 [1,\infty)} \ln(1/x ) } \right)  \cdot \exp \left(  \sqrt{c_1 ||f'||_{L_2 [1,\infty)} \ln(1/x ) } \right) \\
& \leq 2^n \cdot x \cdot \exp \left( n \sqrt{c_1 ||f'||_{L_2 [1,\infty)} \ln(1/x ) } \right),
\end{align*}
where we used in the last step that 
\begin{align*}
 \ln \left(    2^{-n} \cdot x^{-1}  \exp \left( - (n-1) \sqrt{c_1 ||f'||_{L_2 [1,\infty)} \ln(1/x ) } \right)    \right)  \leq \ln (1/x).
\end{align*}
Combining (\ref{induct}) and (\ref{gnT}) with equation (\ref{endpoint1}) and an $n(T)$-times iteration of (\ref{goalset}) yields
\begin{align*}
 \PRO & (X(t) \leq 1 + f(t), \text{  } 0 \leq t \leq T)  \leq g_0 (T)\\
&\leq H^{n(T)}_{2} \Bigl( 2 g_{n(T)} (T) \Bigr)   \leq  g_{n(T)} (T) \cdot 2^{n(T)} \exp \left( n(T) \sqrt{ c_1 ||f'||_{L_2 [1,\infty)}\ln(1/g_{n(T)} (T) )}  \right)   \\
&= T^{-\delta + o(1)}, 
\end{align*}
which completes the proof.
\medskip

\begin{rem}\label{rem:detailsnegjumps}
Let us come back to the discussion about the assumption of the negative jumps in Theorem \ref{mblp}.  The negative jumps are required (Step 4 in the proof) in order to show that (\ref{mblpass}) implies
\begin{align}\label{eqn:lnT52}
 \PRO (X(t) \leq 1 + (\ln T)^{5}, \text{  } 1 \leq t \leq T) \leq  T^{-\delta +o(1)}. 
\end{align}
Alternatively, this can be proved under different assumptions as mentioned in Remark \ref{rem:negjumps}.

On the one hand, with the help of \cite{SupLP}, we require -- instead of the negative jumps -- the assumption (a) in Remark \ref{rem:negjumps}. That means the renewal function $U$ of the ladder height process $H$ satisfies  $U((\ln T)^5) \leq T^{o(1)}$.

On the other hand, one can estimate (\ref{eqn:lnT52}) as follows: For every $T_0 \in (1,T^{o(1)})$, Lemma \ref{associated} and the stationary and independent increments yield 
\begin{align*}
\PRO & (X(t) \leq 1, \text{  } 1 \leq t \leq T) \\
& \geq \PRO (X(t) \leq 1, \text{  } 1 \leq t \leq T_0)\cdot \PRO  (X(T_0) \leq - (\ln T)^5, X(t) \leq 1, \text{  } T_0 \leq t \leq T)\\
& \geq \PRO (X(t) \leq 1, \text{  } 1 \leq t \leq T_0)\\
&\quad \quad \cdot \PRO  (X(T_0) \leq - (\ln T)^5, X(t) -X(T_0) \leq 1 + (\ln T)^5, \text{  } T_0 \leq t \leq T)\\
& \geq \PRO (X(t) \leq 1, \text{  } 1 \leq t \leq T_0)\cdot \PRO  (X(T_0) \leq - (\ln T)^5 ) \\
& \quad \quad \cdot \PRO ( X(t) \leq 1, \text{  } 0 \leq t \leq 1) \cdot  \PRO ( X(t) \leq 1 + (\ln T)^5, \text{  } 1 \leq t \leq T-T_0).
\end{align*}
Thus, using (\ref{mblpass}) leads to 
\begin{align*}
 \PRO (X(t) \leq 1 + (\ln T)^{5}, \text{  } 1 \leq t \leq T) \leq T^{-\delta +o(1)} \cdot  \PRO \left( X(T_0) \leq - (\ln T)^{5}\right)^{-1}.
\end{align*}
 Hence, -- instead of the negative jumps -- it is sufficient for (\ref{eqn:lnT52}) to require the assumption (b) in Remark \ref{rem:negjumps}. That means that there is a $T_0 \in (1,T^{o(1)})$ depending on $T$  such that 
 \begin{align*}
 \PRO (X(T_0) \leq - (\ln T)^{5}) \geq T^{o(1)}.
 \end{align*}
 
Particularly, both assumptions are satisfied by spectrally positive L\'evy processes -- these processes have no negative jumps -- belonging to the domain of attraction of a strictly stable L\'evy process with index $\alpha \in (1,2)$ and skewness parameter $\beta= +1$ (for this case see also \cite{DonRiv}, Theorem 3).

\end{rem}

%%%%%%%%%%%%%%%%%%%%%%%%%%%%%%%%%%%%%%%%%%%%%%%%%%%%%%%%%%%%%%%%%%%%%%%%%%%%%%%%%%%%%%%%%%%%%%%%%%%%%%%%%%%%%%%%%%%%%%%%%%%%%%%%%%%%%%%%%%%%%%%%%%%%%%%

\noindent {\bf Acknowledgement:} Frank Aurzada and Tanja Kramm were supported by the DFG Emmy Noether programme

\bibliographystyle{abbrv}

\end{document}